%% file: convergence_ELLAM_Peaceman.tex
\documentclass[final,reqno]{amsart}

\usepackage{caption}
\usepackage{amsfonts,latexsym,graphicx,amsmath,amssymb,bbm,enumitem,citesort}
\makeindex

\usepackage[color,notref,notcite]{showkeys}

\definecolor{labelkey}{rgb}{0.6,0,1}

\input{macros.tex}

\begin{document}

	\title[Analysis of a family of ELLAM for miscible displacement]{Convergence analysis of a family of ELLAM schemes for a fully coupled model of miscible displacement in porous media}
	
\author{Hanz Martin Cheng}
\address{School of Mathematical Sciences, Monash University, Clayton, Victoria 3800, Australia.
\texttt{hanz.cheng@monash.edu}}
\author{J\'er\^ome Droniou}
\address{School of Mathematical Sciences, Monash University, Clayton, Victoria 3800, Australia.
\texttt{jerome.droniou@monash.edu}}
\author{Kim-Ngan Le}
\address{School of Mathematics and Statistics, The University of New South Wales,
Sydney 2052, Australia.
\texttt{n.le-kim@unsw.edu.au}}

	\date{\today}
	
	%
	%
	
	\maketitle
\begin{abstract}
	We analyse the convergence of numerical schemes in the GDM--ELLAM (Gradient Discretisation Method--Eulerian Lagrangian Localised Adjoint Method) framework for a strongly coupled elliptic-parabolic PDE which models miscible displacement in porous media. These schemes include, but are not limited to Mixed Finite Element--ELLAM and Hybrid Mimetic Mixed--ELLAM schemes. A complete convergence analysis is presented on the coupled model, using
only weak regularity assumptions on the solution (which are satisfied in practical applications),
and not relying on $L^\infty$ bounds (which are impossible to ensure at the discrete level given
the anisotropic diffusion tensors and the general grids used in applications).
\end{abstract}
\section{Model and assumptions}

	We consider the following coupled system of partial differential equations, modelling the miscible displacement
of one fluid by another in a porous medium:
	\begin{subequations}\label{eq:model}
		\begin{equation} \label{pressure}
		\begin{aligned}
		\div\Daru &= q^{+}-q^{-}\qquad \mbox{ on } Q:=\O \times (0,T) \\
		\Daru &= - \dfrac{\mathbf{K}}{\mu(c)} \nabla p \qquad \mbox{ on } Q\\
		\end{aligned}
		\end{equation}
		\begin{equation} \label{concentration}
		\phi \dfrac{\partial c}{\partial t} + \div(\Daru c-\textbf{D}(\textbf{x},\Daru)\nabla c) +q^-c= q^{+} \qquad \mbox{ on } Q \\
		\end{equation}
		with unknowns  $p(\x,t), \darcyU(\x,t),$ and $c(\x,t)$ which denote the pressure of the mixture, the Darcy velocity, and the concentration of the injected solvent, respectively.
		The functions $q^{+}$ and $q^{-}$ represent the injection and production wells respectively, and $\textbf{D}(\textbf{x},\Daru)$ denotes the diffusion--dispersion tensor
		\begin{equation} \nonumber
		\textbf{D}(\textbf{x},\Daru) = \phi(\textbf{x})\left[d_{m}\textbf{I}+d_{l}|\Daru|\proj(\Daru)+d_{t}|\Daru|\left(\textbf{I}-\proj(\Daru)\right)\right]
		\mbox{ with }\proj(\Daru) = \left(\dfrac{u_{i}u_{j}}{|\Daru|^{2}}\right)_{i,j}.
		\end{equation}
		Here, $d_{m}>0$ is the molecular diffusion coefficient, $d_{l}>0$ and $d_{t}>0$ are the longitudinal and transverse dispersion coefficients respectively, and $\proj(\Daru)$ is the projection matrix along the direction of $\Daru$.
		Also, $\K$ is the symmetric, bounded uniformly coercive diffusion tensor, and $\mu(c)=\mu(0)[(1-c)+M^{1/4}c]^{-4}$ is the viscosity of the fluid mixture, where $M=\mu(0)/\mu(1)$ is the mobility ratio of the two fluids.
		As usually considered in numerical tests, we take no-flow boundary conditions:
		\begin{equation}
		\Daru \cdot \bfn = (\textbf{D}\nabla c) \cdot \bfn = 0 
		\mbox{ on }\partial\O \times (0,T).
		\end{equation}
	\end{subequations}
The concentration equation is completed by an initial condition, and the pressure equation by
an average condition:
\[
c(\x,0)=c_{\rm ini} \mbox{ for all $\x\in \O$, }\int_\O p(\x,t)d\x=0\mbox{ for all $t\in (0,T)$}.
\]

Problem \eqref{eq:model} is used in enhanced oil recovery. Exact solutions of this model
are usually inaccessible, especially with data as encountered in applications; thus
the design and convergence analysis of numerical schemes for \eqref{eq:model} is therefore of particular importance. The main purpose of this work is to provide a GDM--ELLAM (Gradient Discretisation Method--Eulerian Lagrangian Localised Adjoint Method) framework for model \eqref{eq:model} and to establish convergence of numerical schemes that fall under this framework.  Some of the schemes covered by this framework are the Mixed Finite Element--ELLAM and Hybrid Mimetic Mixed--ELLAM schemes. An overview of studies and analysis involving ELLAM schemes is presented in \cite{RC-02-overview}. Convergence analysis was performed for MFEM--ELLAM schemes (or similar) in \cite{AW95-CMM,W08-errorEstimate-ELLAM}. We note here that \cite{AW95-CMM} only considers the concentration equation \eqref{concentration}
(assuming that $\darcyU$ is given), whereas \cite{W08-errorEstimate-ELLAM} provides error estimates for the complete coupled model \eqref{eq:model}. However, these analysis were carried out under restrictive regularity assumptions on the porosity $\phi$ and on the solution $(p,\darcyU,c)$ to the model; in particular, the minimal
assumptions in \cite{W08-errorEstimate-ELLAM} are $c\in H^1(0,T;H^2(\O))\cap L^\infty(0,T;W^{2,r}(\O))$ (for $r>2$)
and $\darcyU\in W^{1,\infty}(\O\times(0,T))$, and \cite{AW95-CMM} supposes that $c,\diffTens\grad c\in C^1(0,T;H^1(\O))$ and $\phi,\darcyU \in W^{1,\infty}(\O \times (0,T))$. However, in reservoir modeling, transitions between different rock layers are usually discontinuous; thus, the permeability may vary rapidly over several orders of magnitude, with local variations in the range of 1mD to 10D, where D is the Darcy unit \cite{LM15-mathematicalModels}. Due to this discontinuity of $\K$, the solutions to \eqref{eq:model} cannot expect to satisfy the regularity conditions stated above. 
Actually, all reported numerical tests \cite{WLELQ-00,fvca8-ellam,CD-07,ckm13} seem to have been on
tests cases for which such regularity of the data and/or the solutions do not hold.

More recent developments of ELLAM techniques involve Volume Corrected Characteristic Mixed Methods (VCCMM), which are, in essence, ELLAM schemes with volume adjustment to achieve local mass conservation. Convergence analysis, as well as stability, monotonicity, maximum and minimum principles for these schemes have been studied in \cite{AW10-convergence,AW11-stability-monotonicity-implementation}. However, these studies only consider a single pure advection
model (that is, \eqref{concentration} with $\diffTens=0$), and assume the regularity $\darcyU\in C^1(\O\times(0,T))$, which, as explained above, is not expected in applications. Without accounting for diffusion, the maximum principle is accessible, and thus the analysis strongly benefits from the resulting $L^\infty$ bounds on the approximate solution. On the contrary, in the presence of anisotropic heterogeneous diffusion $\K$ and $\diffTens(\darcyU)$, and on grids as encountered in applications, constructing schemes that satisfy the maximum principle is extremely difficult -- to
this day, only \emph{nonlinear} schemes are known to preserve the maximum principle in general, and even these do not necessarily have nice coercivity features \cite{D14-FVschemes}.

As a matter of fact, the convergence analysis of numerical approximations of \eqref{eq:model} under weak regularity assumptions has recently received an increasing interest; see, e.g., \cite{CD-07,ckm15} for finite volume
methods and \cite{rivwalk11,GLR17-conv-DG} for discontinuous Galerkin methods. It therefore seems natural
to consider doing such an analysis for characteristic-based discretisation of the advection term.
This leaves open the choice of particular discretisations of the diffusion terms in the model.
Instead of selecting one particular discretisation of these terms, we work inside a framework that
enables a simultaneous analysis of various such discretisations.

The Gradient Discretisation Method is a generic framework to discretise diffusion equations \cite{GDMBook16}.
It consists in replacing, in the weak formulation of the equation, the continuous space and functions/gradients
by a discrete space and reconstructions of functions/gradients. This space together with the
reconstruction operators are called a gradient discretisation (GD). The convergence of the resulting scheme is ensured
under a few properties (3 or 4, depending on the non-linearities in the model) on the GDs.
The efficiency of the GDM is found in its flexibility: various choices of GDs lead to various
classical methods (conforming and non-conforming finite elements, finite volumes, etc.),
which means that the analysis carried out in the GDM directly applies to all these
methods at once.

The main contributions of this work are
\begin{itemize}
		\item  Presentation of a GDM--ELLAM framework for the complete coupled model \eqref{eq:model}, which covers a variety of discretisations of the diffusion terms.
	\item Convergence analysis that only relies on energy estimates based on coercivity
but not maximum principle, and is therefore adapted to anisotropic heterogeneous diffusion
on generic grids as encountered in applications.
	\item Analysis carried out under weak regularity assumptions on the data, as seen in
previously reported numerical tests on various schemes for \eqref{eq:model}.
\end{itemize}

The paper is made up of two main components, a conclusion and an appendix. The first main component (Sections \ref{sec:GDM-intro} to \ref{sec:examples}) focuses on the presentation of the GDM--ELLAM, the main convergence result, and numerical schemes that fall into this framework; whereas the second component (Sections \ref{sec:flow} to \ref{sec:comp-conv}) establishes properties on the flows, \emph{a priori} estimates on the solution
to the scheme, and prove its convergence by using compactness techniques.
The conclusion (Section \ref{sec:conclusion}) recalls the main elements of the paper, and the appendix (Section \ref{sec:appen:GS}) contains a few technical compactness results.

\medskip

We start by presenting the weak formulation of the model \eqref{eq:model}. This is followed by Section \ref{sec:GDM-intro}, which gives a short overview of the gradient discretisation method, together with some standard properties which ensure the convergence of the gradient schemes. Section \ref{sec:GDM--ELLAM-scheme} then presents the GDM--ELLAM for the model \eqref{eq:model}, followed by the main results: existence and uniqueness of the solution to the scheme, and its convergence to the weak solution of \eqref{eq:model} under weak regularity assumptions. Section \ref{sec:examples} then gives some of the numerical schemes that are covered by the GDM--ELLAM framework, together with proofs on why they satisfy the regularity assumptions. 

\medskip

Since ELLAM schemes are based on characteristic methods, we need to solve characteristics along which the solution flows. The properties of the flow (described by the characteristics), together with some estimates that come with it, are presented in Section \ref{sec:flow}. These properties are not trivial to establish due to the weak regularity assumptions. A priori estimates are then obtained in Section \ref{sec:a-priori-est}, which lead us to compactness arguments that will help establish the proof of convergence. Finally, we prove our convergence result in Section \ref{sec:comp-conv}. The ELLAM discretisation of the advection term makes the energy estimates and the
convergence analysis of the corresponding terms rather tricky.
The fine results from Sections \ref{sec:flow} and \ref{sec:a-priori-est} are instrumental to obtain the major estimates and the proper convergence of the advection term. We also note that, at the core of our convergence analysis
lies some generic compactness results of \cite{GDMBook16}, which are flexible enough to be used even outside
a purely GDM framework (as in the GDM--ELLAM framework here).
 
\medskip

Throughout the article we assume the following properties, satisfied by $\diffTens$, $\K$ and $\mu$ previously
described.
\begin{subequations}\label{assump.global}
\be
c_{\rm ini}\in L^\infty(\O)\mbox{ and }
q^+,q^-\in L^\infty(\O\times(0,T)) \mbox{ with $|q^+|\leq M_{q^+}$, $|q^-|\leq M_{q^-}$}.\label{hyp:q.bounded}
\ee
\be
\begin{aligned}
&\mbox{$\phi$ is piecewise smooth on a mesh, and there exists $\phi_*,\phi^*>0$ such that}\\
& \phi_* \leq \phi \leq \phi^*\mbox{ on $\O$}.\label{hyp:phi}
\end{aligned}
\ee

\be
\begin{aligned}
	&\mbox{$A=\textbf{K}/\mu$ is Carath\'eodory and there exists $\alpha_A$ and $\Lambda_A$ s.t. for a.e.\ $\x\in\O$,}\\
	&\forall (s,\xi)\in\R\times \R^d\;:\; A(\x,s)\xi \cdot \xi \geq \alpha_A|\xi|^2
	\mbox{ and }|A(\x,s)|\leq \Lambda_A.\label{hyp:viscosity}
\end{aligned}
\ee
\be
\begin{aligned}
&\mbox{$\diffTens$ is Carath\'eodory and there exists $\alpha_\diffTens$ and $\Lambda_\diffTens$ s.t.
for a.e.\ $\x\in\O$,}\\
&\forall \xi, \zeta \in \R^d\;:\; \diffTens(\x,\zeta)\xi \cdot \xi \geq \alpha_\diffTens (1+|\zeta|)|\xi|^2
\mbox{ and }|\diffTens(\x,\zeta)|\leq \Lambda_\diffTens(1+ |\zeta|).\label{hyp:diff}
\end{aligned}
\ee
\end{subequations}
Here, ``Carath\'eodory'' means measurable with respect to $\x$ and continuous with respect
to the other variables.
In \eqref{hyp:phi} as well as \ref{hyp:hdiv} below, ``mesh'' is to be understood in the simplest intuitive
way: a partition of $\O$ into polygonal (in 2D) or polyhedral (in 3D) sets.
Under these assumptions, we consider the following standard notion of weak solution to \eqref{eq:model}
(see, e.g., \cite{F95}).

\begin{definition}[Weak solution to the miscible displacement model]
A couple $(p,c)$ is a weak solution of \eqref{eq:model} if
\begin{equation}\label{press.weak} 
\begin{aligned}
&p\in L^\infty(0,T;H^1(\O))\,,\;\int_\O p(\x,t) d\x=0 \mbox{ for a.e.\ $t\in (0,T)$, and}\\
&\int_0^T\int_{\O} \frac{\K(\x)}{\mu(c(\x,t))} \grad p(\x,t) \cdot \grad \psi(\x,t) d\x\d t\\
&\qquad\qquad= \int_0^T\int_{\O}  (q^{+}(\x,t) -q^{-}(\x,t)) \psi(\x,t) d\x dt\,, \quad \forall \psi\in C^\infty(\overline{\O}\times [0,T]),
\end{aligned}
\end{equation}
and, setting $\darcyU(\x,t)=-\frac{\K(\x)}{\mu(c(\x,t))}\nabla p(\x,t)$,
	\begin{equation}\label{conc.weak}
\begin{aligned}
&c\in L^2(0,T;H^1(\O))\,,\; (1+|\Daru|)^{1/2}\nabla c\in L^2(\O\times(0,T))^d\,,\\
&-\int_\O \phi(\x)c_{\rm ini}(\x)\varphi(\x,0) d\x
-\int_0^T\int_{\Omega} \phi(\x) c(\x,t)\dfrac{\partial \varphi}{\partial t}(\x,t) d\x dt\\
&\quad +\int_0^T\int_\O \diffTens(\x,\Daru(\x,t))\nabla c(\x,t) \cdot \nabla \varphi(\x,t) d\x dt\\
&\quad-\int_0^T\int_{\Omega} c(\x,t)\darcyU(\x,t) \cdot \nabla \varphi(\x,t)d\x dt
+\int_0^T\int_{\Omega} q^-(\x,t)c(\x,t)\varphi(\x,t) d\x dt\\
&\quad =\int_0^T\int_{\Omega} q^+(\x,t)\varphi(\x,t)d\x dt\,, \quad \forall \varphi \in C^\infty_c(\overline{\O}\times [0,T)).
\end{aligned}
	\end{equation}
\end{definition}

\begin{remark}[Injection concentration and gravity]
The model \eqref{eq:model} assumes an injection concentration of
$1$ and neglects the gravity effects. A generic injection concentration
$\widehat{c}$ could be considered upon the trivial modification $q^+\leadsto \widehat{c}q^+$ in \eqref{concentration}.
To include gravity effect, we would have to set
$\darcyU=-\frac{\K}{\mu(c)}(\nabla p-\rho(c)\mathbf{g})$ (with $\rho$ a continuous function).
The analysis we conduct thereafter can easily be adapted to both changes.
\end{remark}

\section{Brief presentation of the gradient discretisation method}\label{sec:GDM-intro}

The gradient discretisation method (GDM) is a discretisation method for diffusion
equations which consists in replacing, in the weak formulation of the PDE, the
continuous space and time operators by discrete counterparts \cite{GDMBook16}. These discrete elements
are given by what is a called a gradient discretisation (GD). The convergence of the resulting
schemes (called gradient schemes (GS)) can be established under a few assumptions on the
gradient discretisations. We give here a brief presentation of GDs and the standard properties that ensure the convergence of the corresponding GSs for standard elliptic and parabolic PDEs.
In the rest of the paper, the notations $L^p(X)$ and $L^p(0,T;L^q(X))$ are sometimes also
used in lieu of $L^p(X)^d$ and $L^p(0,T;L^q(X))^d$.

\begin{definition}[Space and space--time gradient discretisations] \label{GDdef}
	A space gradient discretisation for no-flow boundary conditions is $\disc=(X_\disc,\PiD,\gradD)$, where
	\begin{itemize}
		\item the set of discrete unknowns $X_{\disc}$ is a finite dimensional vector space on $\R$
		\item the function reconstruction $\PiD : X_{\disc} \rightarrow L^{\infty}(\O)$ is linear
		\item the gradient reconstruction $\gradD: X_{\disc}\rightarrow L^{\infty}(\O)^d$ is linear.
	\end{itemize}
The operators $\PiD$ and $\gradD$ must be chosen so that
\begin{equation}\nonumber
\norm{v}{\disc} := \left( \norm{\gradD v}{L^2(\O)}^2+ \left|\int_\O \PiD v(\x)d\x \right|^2 \right)^{\frac{1}{2}}
\end{equation}
is a norm on $X_\disc$.

A space--time gradient discretisation is $\disc^T=(\disc,\ICinterp_{\disc},(t^{(n)})_{n=0,\dots, N})$
such that $\disc$ is a space GD, $0=t^{(0)}<\dots <t^{(N)}=T$ are time steps, and
$\ICinterp_{\disc}:L^\infty(\O)\to X_\disc$ is an operator used to interpolate initial conditions
onto the unknowns.
\end{definition}

Considering for example \eqref{press.weak} and replacing the space $H^1(\O)$ with $X_\disc$,
the functions by reconstructions using $\Pi_\disc$ and the gradients by reconstructions $\nabla_\disc$,
we obtain the corresponding gradient scheme \eqref{GSpress}.
The simplest example of a GD can be described by considering $\mathbb{P}_1$ finite elements
on a simplicial mesh. A vector $v\in X_{\disc}$ is made of vertex values $(v_\vertex)_{\vertex\mbox{ \scriptsize vertex of the mesh}}$, $\Pi_\disc v$ is the continuous, piecewise linear reconstruction from these values,
and $\nabla_\disc v$ is the standard gradient of this reconstruction. Other examples of GDs
are given in Section \ref{sec:examples}.

The accuracy of a GD and convergence properties of the resulting GS are measured through three parameters,
that respectively correspond to a discrete Poincar\'e--Wirtinger constant,
an interpolation error, and a measure of defect of conformity (error in a discrete
Stokes formula):
\begin{align}
&C_\disc = \max_{v\in X_\disc}  \dfrac{\norm{\PiD v}{L^2(\O)}}{\norm{v}{\disc}},\label{def:CD}\\
&\forall \phy\in H^1(\O)\,,\;
S_\disc(\phy)=\min_{v\in X_\disc}(\norm{\PiD v- \phy}{L^2(\O)}+\norm{\gradD v -\grad \phy}{L^2(\O)}),\nonumber\\
&\forall \bphi \in H_\div(\O)\,,\;
 W_{\disc}(\bphi) = \max_{v\in X_\disc \backslash \{ 0 \}} \frac{
\dsp\left|\int_{\O} 
\left(\gradD v(\x) \cdot \bphi(\x) + \PiD v(\x) \divg \bphi(\x)\right) d\x\right|
}{\norm{v}{\disc}}.\nonumber
\end{align}

\begin{definition}[Properties of GDs]\label{def:propGDs}
A sequence $(\disc_m)_{m\in\N}$ of space gradient discretisations is
\begin{itemize}
\item coercive if there exists $C_p \in \R_+$ such that $C_{\disc_m}\leq C_p$ for all $m\in\N$,
\item GD-consistent if, for all $\phy\in H^1(\O)$, $S_{\disc_m}(\phy)\to 0$ as $m\to\infty$,
\item limit-conforming if, for all $\bphi\in H_\div(\O)$, $W_{\disc_m}(\bphi)\to 0$ as $m\to\infty$,
\item compact if for any sequence $v_m \in X_{\disc_m}$ such that $(\norm{v_m}{\disc_m})_{m\in\N}$ is bounded, the sequence $(\Pi_{\disc_m}v_m)_{m\in\N}$ is relatively compact in $L^2(\O)$.
\end{itemize}
A sequence of space--time gradient discretisations $(\disc_m^T)_{m\in\N}$ is coercive, limit-confor\-ming
or compact if its underlying sequence of space gradient discretisations satisfy the corresponding property.
Finally, $(\disc_m^T)_{m\in\N}$ is GD-consistent if the underlying sequence of spatial GDs is GD-consistent and if
\begin{itemize}
\item with $\dtDisc_m=t_m^{(n+1)}-t_m^{(n)}$, $\max_{n=0,\ldots,N_m-1}\dtDisc_m\to 0$ as $m\to\infty$,
\item for all $\phy\in L^\infty(\O)$, $(\Pi_{\disc_m}\ICinterp_{\disc_m}\phy)_{m\in\N}$
is bounded in $L^\infty(\O)$ and converges to $\phy$ 
in $L^2(\O)$ as $m\to\infty$.
\end{itemize}
\end{definition}

\begin{remark} Actually, the limit-conformity or compactness of a sequence of space GDs implies its
coercivity. The latter is however explicitly mentioned as a bound on $C_{\disc_m}$ is
useful throughout the analysis. 

In the GDM, the interpolant $\ICinterp_\disc$ is usually defined on $L^2(\O)$;
in the context of Problem \eqref{eq:model}, the initial condition
is always assumed to be bounded and it is therefore natural to only consider
interpolants of initial conditions in $L^\infty(\O)$.
\end{remark}

If $\disc$ is a space GD, $0=t^{(0)}<\cdots<t^{(N)}=T$ are time steps
and $z=(z^{(n)})_{n=0,\ldots,N}\in X_{\disc}^{N+1}$, we define the space--time reconstructions
$\Pi_\disc z\in L^\infty(\O\times(0,T))$, $\widetilde{\Pi}_\disc z \in L^\infty(\O\times(0,T))$
and $\nabla_\disc z\in L^\infty(\O\times(0,T))^d$
by
\[
\begin{aligned}
&\forall n=0,\ldots,N-1\,,\;\forall t\in (t^{(n)},t^{(n+1)}]\,,\mbox{ for a.e.\ $\x\in\O$}\,,\\
&\Pi_\disc z(\x,t)=\Pi_\disc z^{(n+1)}(\x)\,,\;\widetilde{\Pi}_\disc z(\x,t)=\Pi_\disc z^{(n)}(\x)\\
&\mbox{and }\nabla_\disc z(\x,t)=\nabla_\disc z^{(n+1)}(\x).
\end{aligned}
\]
\section{GDM--ELLAM scheme and main result}\label{sec:GDM--ELLAM-scheme}

The diffusion terms in \eqref{pressure} and \eqref{concentration} are discretised by the gradient discretisation
method. This enables us to carry out a unified convergence analysis for many different numerical discretisations
of these diffusion terms.
There are grounds for
considering possibly different GDs for each equation in \eqref{eq:model} (see e.g. Section \ref{sec:conf.mixed}).
We therefore take a space gradient discretisation $\discP  = ( X_{\discP}, \PiDp, \gradDp)$
for the pressure, and a space--time gradient discretisation $\discC^T=(\discC,\ICinterp_{\disc},(t^{(n)})_{n=0,\dots, N})$ for the concentration.

From here onwards, the variables $\x$ and $t$ may be dropped in the integrals when there are no risks of confusion.

\begin{definition}[GDM--ELLAM scheme]
The gradient scheme for \eqref{eq:model} reads as: find $(p^{(n)})_{n=1,\ldots,N}\in X_{\discP}^N$
and $(c^{(n)})_{n=0,\ldots,N}\in X_{\discC}^{N+1}$ such that $c^{(0)}=\ICinterp_{\discC} c_{\rm ini}$
and, for all $n=0,\ldots,N-1$,
\begin{itemize}
\item[i)] $p^{(n+1)}$ solves
 \begin{equation}\label{GSpress}
\begin{aligned}
& \int_{\O} \PiDp p^{(n+1)} =0 \mbox{ and }\\
&\int_{\O} A(\x,\PiDc c^{(n)}) \gradDp p^{(n+1)} \cdot \gradDp z = \int_{\O}  (q^{+}_n -q^{-}_n) \PiDp z\,,\qquad
\forall z \in X_{\discP}
\end{aligned}
 \end{equation}
where $q^\pm_n(\cdot)=\frac{1}{\dtDisc}\int_{t^{(n)}}^{t^{(n+1)}}q^\pm(\cdot,s)ds$
(or, alternatively, $q^\pm_n=q^\pm(t^{(n)})$ if $q^\pm$ are continuous in time).

\item[ii)] A Darcy velocity $\discDarcyU$ is reconstructed from $p^{(n+1)}$
and, to account for the advection term in the concentration equation,
the following advection equation is considered; it defines space-time
test functions from chosen final values:
\begin{equation}\label{Advection}
\phi \partial_t v+\discDarcyU \cdot \nabla v =0 \quad \text{on } (t^{(n)},t^{(n+1)})\,,\mbox{ with $v(\cdot,t^{(n+1)})$ given}.
\end{equation}

\item[iii)] Setting $\U_\discP^{(n+1)} = A(\x, \PiDc c^{(n)}) \gradDp p^{(n+1)}$ and
using a weighted trapezoid rule with weight $w\in [0,1]$ for the time-integration of the source term, 
$c^{(n+1)}$ satisfies
	\begin{equation}\label{GSconc}
	\left\{	\begin{aligned}
	&\mbox{For all $z \in X_{\discC}$, setting $v$ the solution to \eqref{Advection} with $v(\cdot,t^{(n+1)})=\PiDc z$,}\\
	&\int_{\O} \phi \PiDc c^{(n+1)} \PiDc z -\int_{\O} \phi \PiDc c^{(n)} v(t^{(n)}) \\
&\qquad +\dtDisc \int_{\O} \diffTens(\x,\U_\discP^{(n+1)}) \gradDc c^{(n+1)} \cdot \gradDc z \\
	&\qquad+ w\dtDisc \int_{\O} \PiDc c^{(n)} v(t^{(n)})q^-_n
+ (1-w)\dtDisc \int_{\O} \PiDc c^{(n+1)} \PiDc z q^-_{n+1} \\
	&\qquad= w\dtDisc \int_{\O} q^+_n v(t^{(n)}) +(1-w) \dtDisc \int_{\O} q^+_{n+1}\PiDc z,
	\end{aligned}\right.
	\end{equation}
where $q^\pm_N=q^\pm_{N-1}$ if these quantities are defined by averages on time intervals
(there is no time interval $(t^{(N)},t^{(N+1)})$).
\end{itemize}
\end{definition}

\begin{remark} Using a weighted trapezoid rule for the time discretisation of the reaction/source
terms is essential to obtain an accurate numerical scheme \cite{AH06,fvca8-ellam}.
\end{remark}

Defining the flow $F_{t}$ such that, for a.e.\ $\x\in\O$,
\begin{equation} \label{charac}
\dfrac{dF_{t}(\x)}{dt} =\dfrac{\discDarcyU(F_{t}(\x))}{\phi(F_t(\x))}\quad\mbox{ for $t\in [-T,T]$}, \qquad F_{0}(\x)= \x,
\end{equation}
the solution to \eqref{Advection} is understood in the sense:
for $t\in (t^{(n)},t^{(n+1)})]$ and a.e.\ $\x\in\O$, $v(\x,t)=v(F_{t^{(n+1)}-t}(\x),t^{(n+1)})$.
Hence, in \eqref{GSconc}, $v(\cdot,t^{(n)})=\PiDc z(F_{\dt^{(n+1/2)}}(\cdot))$.
Under Assumptions \eqref{hyp:phi} and \ref{hyp:hdiv} below, the existence and uniqueness of the
flow is discussed in Lemma \ref{lem:def.flow}.
We note that $F_t$ depends on $n$ through $\discDarcyU$, but this dependency is not explicitly indicated
when there is no risk of confusion.

\medskip

The convergence theorem is established under the following assumptions.
We show in Section \ref{sec:examples} that various finite element and finite volume
methods are given by GDs that satisfy these assumptions.

\begin{enumerate}[label=\combineln{A}{\arabic*}]

\item \label{hyp:GDs} $(\discP_m)_{m\in\N}$ and $(\discC^T_m)_{m\in\N}$ are coercive, GD-consistent
and limit-conforming sequences of space--time GDs, and $(\discC^T_m)_{m\in\N}$ is moreover compact.
Denoting by $0=t_m^{(0)}<\cdots<t_m^{(N_m)}=T$ the time steps of $\discC_m$,
it is
assumed that there exists $M_t\ge 0$ such that, for all $m\in\N$ and $n=1,\ldots,N-1$,
$\dt_m^{(n+1/2)}\le M_t \dt_m^{(n-1/2)}$.

\item \label{hyp:PiD.flow} There exists $M_F\ge 0$ such that, for all $m\in\N$, $z\in X_{\discC_m}$, all $n=0,\ldots,N_m-1$, and all $s\in [-T,T]$,
\[
\norm{\Pi_{\discC_m} z(F_s)-\Pi_{\discC_m} z}{L^1(\O)}\le M_F |s| \norm{\discDarcyUm}{L^2(\O)}
\norm{\nabla_{\discC_m} z}{L^2(\O)}.
\]

\item \label{hyp:smooth.interp} For all $m\in\N$ there is an interpolant
$\sinterp_{\discC_m}:C^\infty(\overline{\O})\to X_{\discC_m}$
such that, for all $\phy\in C^\infty(\overline{\O})$,
as $m\to\infty$, $\grad_{\discC_m}\sinterp_{\discC_m} \phy \to \grad \phy$ in $L^4(\O)^d$
and $\Pi_{\discC_m}\sinterp_{\discC_m}\phy\to\phy$ in $L^\infty(\O)$.

\item \label{hyp:hdiv} There exists $M_\div>0$ such that, for all $m\in\N$ and $n=0,\ldots,N_m-1$, $\discDarcyUm\in H_{\div}(\O)$ is piecewise polynomial on a mesh,
$\discDarcyUm\cdot\bfn=0$ on $\partial\O$, and $|\divg \discDarcyUm| \leq M_\div$ on $\O$.

\item \label{hyp:conv.u} 
If $(p_m,c_m)\in X_{\discP_m}^{N_m}\times X_{\discC_m}^{N_m+1}$ is a solution to
the GDM--ELLAM scheme with $(\discP,\discC^T)=(\discP_m,\discC^T_m)$ 
and $\discDarcyUm[]:\O\times (0,T)\to\R^d$ is defined by $\discDarcyUm[](\cdot,t)=\discDarcyUm$
for all $t\in (t_m^{(n)},t_m^{(n+1)})$ and $n=0,\ldots,N_m-1$, then, when \ref{hyp:GDs}--\ref{hyp:hdiv} hold:
\begin{enumerate}
\item\label{hyp:conv.u.1} $\norm{\discDarcyUm[]}{L^\infty(0,T;L^2(\O))}\le C_m\norm{\nabla_{\discP_m}p_m}{L^\infty(0,T;L^2(\O))}$ with $(C_m)_{m\in\N}$ boun\-ded.
\item\label{hyp:conv.u.2} if $p\in L^2(0,T;H^1(\O))$ and $c\in L^2(\O\times(0,T))$ are such that, as $m\to\infty$, $\Pi_{\disc_m}p_m\to \overline{p}$, $\Pi_{\discC_m}c_m\to c$ and $\nabla_{\discP_m}p_m\to \nabla p$
in $L^2(\O\times(0,T))$, then $\mathbf{u}_{\discP_m}\to \Daru=-\frac{\K}{\mu(c)}\nabla p$ weakly in $L^2(\O\times(0,T))^d$.
\end{enumerate}
\end{enumerate}

\begin{theorem}[Convergence of the GDM--ELLAM scheme]\label{th:main.convergence}
Under Assumptions \eqref{assump.global} and \ref{hyp:GDs}--\ref{hyp:conv.u}, for any $m\in\N$ there is a unique
$(p_m,c_m)\in X_{\discP_m}^{N_m}\times X_{\discC_m}^{N_m+1}$ solution of the GDM--ELLAM scheme with $(\discP,
\discC^T)=(\discP_m,\discC_m^T)$.
Moreover, up to a subsequence as $m\to\infty$,
\begin{itemize}
\item $\Pi_{\discP_m} p_m \conv p$ and $\nabla_{\discP_m}p_m\to \nabla p$
weakly-$*$ in $L^\infty(0,T;L^2(\O))$ and strongly in $L^r(0,T;L^2(\O))$ for all $r<\infty$,
\item $\Pi_{\discC_m} c_m \conv c$ weakly-$*$ in $L^\infty(0,T;L^2(\O))$ and strongly
in $L^r(0,T;L^2(\O))$ for all $r<\infty$,
\item $\nabla_{\discC_m}c_m\to \nabla c$ weakly in $L^2(\O\times(0,T))^d$,
\end{itemize}
where $(p,c)$ is a weak solution of \eqref{eq:model}.
\end{theorem}

\begin{remark}[About the assumptions]
Assumption \ref{hyp:GDs} is standard in analysis of gradient schemes, except for the assumption
on the time steps, which is not very restrictive in practice (it is for example satisfied
by uniform time steps, used in most numerical tests on \eqref{eq:model}, see e.g.\ \cite{WLELQ-00,CD-07}).
Assumption \ref{hyp:PiD.flow} is probably the most technical to check for specific methods;
we however provide two results (Lemmas \ref{lem:flow.cont} and \ref{lem:pwcst.flow})
which show that it is satisfied for a wide range of conforming or non-conforming methods.
Assumption \ref{hyp:smooth.interp} is satisfied by all standard interpolants associated
with numerical methods for diffusion equations.
Assumption \ref{hyp:hdiv} is natural given the pressure equation \eqref{pressure} and 
the boundedness assumption \eqref{hyp:q.bounded} on $q^+$ and $q^-$.
Finally, Assumption \ref{hyp:conv.u} is also rather natural since it is expected that
the reconstructed Darcy velocity $\discDarcyU[]$ is closely related to the reconstructed
concentration $\PiDc c$ and pressure gradient $\nabla_\discP p$.
\end{remark}

\begin{remark}[One GD per time step]\label{rk:various.Dp}
In some particular cases, most notably the discretisation via mixed finite elements (see Section \ref{sec:conf.mixed}),
the gradient discretisation $\discP$ changes with each time step. Each equation
\eqref{GSpress} is written with a specific gradient discretisation $\discP^{(n+1)}$.
Hence, the choice $\discP$ of a gradient discretisation for the pressure actually amounts
to choosing a family $\discP=(\discP^{(i)})_{i=1,\ldots,N}$.
Theorem \ref{th:main.convergence} remains valid provided that the coercivity,
GD-consistency and limit-conformity of a sequence $(\discP_m)_{m\in\N}=((\discP_m^{(i)})_{i=1,\ldots,N_m})_{m\in\N}$
of such families of GDs are defined as in Definition \ref{def:propGDs} with
\[
C_{\discP_m}=\max_{i=1,\ldots,N_m} C_{\discP_m^{(i)}}\,,\quad
S_{\discP_m}=\max_{i=1,\ldots,N_m} S_{\discP_m^{(i)}}\quad\mbox{and}\quad
W_{\discP_m}=\max_{i=1,\ldots,N_m} W_{\discP_m^{(i)}}.
\]
\end{remark}

\section{Sample methods covered by the analysis}\label{sec:examples}

The ELLAM is a way to deal with the advection term in the concentration equation.
Various numerical methods can be chosen to discretise the diffusion terms in this equation,
as well as in the pressure equation. These methods correspond to selecting specific
gradient discretisations $\discC$ and $\discP$. Here, we detail some of the GDs
corresponding to methods used in the literature in conjunction with the ELLAM, and we show that
they all satisfy the assumptions of Theorem \ref{th:main.convergence}. As a consequence, our convergence
result applies to all these methods.

In the following, for simplicity of notations, we drop the index $m$ in the gradient discretisations
and we consider Assumptions \ref{hyp:GDs}--\ref{hyp:conv.u} `as the mesh size and time step go to zero'
(as opposed to `as $m\to\infty$').

\subsection{Conforming/mixed finite-element methods}

When discretising the model \eqref{eq:model} using finite element methods for the diffusion
terms and the ELLAM for the advection term, it is
natural to use a mixed method for the pressure equation and a conforming method
for the concentration equation. The mixed method provides an appropriate Darcy velocity
that can be used to build the ELLAM characteristics. This approach was
considered in \cite{WLELQ-00,W08-errorEstimate-ELLAM}.
We show here that such a mixed/conforming FE--ELLAM scheme fits into our GDM--ELLAM framework,
so that the convergence result of Theorem \ref{th:main.convergence} applies to
the schemes in the aforementioned references. Notice that, contrary to the convergence
analysis done for example in \cite{W08-errorEstimate-ELLAM}, our convergence result relies
on very weak regularity assumptions on the data and solution, that are usually satisfied
in practical applications.

\subsubsection{Description of the conforming and mixed FE GDs}\label{sec:conf.mixed}

Any conforming Galerkin approximation, which include conforming finite element methods (such as $\mathbb{P}_k$
FE on simplices, or $\mathbb{Q}_k$ FE on Cartesian grids), fits into the GDM framework.
A finite-dimensional subspace $V_h$ of $H^1(\O)$ being chosen,
we define $(X_\discC,\Pi_\discC,\nabla_\discC)$ by $X_\discC=V_h$ and, for $v\in V_h$,
$\Pi_\discC v=v$ and $\nabla_\discC v=\nabla v$. The interpolant $\ICinterp_\discC$ can be either
chosen as the orthogonal projection on $V_h$, in the case of an abstract space, or
as the standard nodal interpolant for specific FE spaces.

\medskip

We now describe a gradient discretisation $\discP$ corresponding to the $\RT0$ mixed finite
element method. The following construction can be extended to higher order $\mathbb{RT}_k$
finite elements \cite{EGH:RTK.GS}.
A conforming simplicial or Cartesian mesh $\mesh$ being chosen, define
\begin{subequations}
\begin{align}
&  \bfVh = \{\bfv \in H_{\div}(\Omega) \,:\, \bfv_{|K} \in \RT0(K)\,,\; \forall K \in \mesh\,,\;
\bfv\cdot\bfn=0\mbox{ on $\partial\O$}\}, \label{spaceVh}\\
&  W_{h} = \{z \in L^2(\Omega)\,:\,  z_{|K} \mbox{ constant}\,,\; \forall K \in \mesh\},  \label{spaceWh}
\end{align}
\label{mix:rtkspaces}\end{subequations}
where $\RT0$ is the lowest order Raviart--Thomas space on the cell $K$ (the description of
$\RT0$ depends if this cell is a simplex or Cartesian cell).
After choosing a diffusion tensor $\mathcal A$ -- that is, a symmetric, uniformly bounded and coercive
matrix-valued function $\O\to M_d(\R)$ -- a gradient discretisation $\discP=(X_\discP,\Pi_\discP,\nabla_\discP)$ is constructed
by setting $X_\discP=W_h$ and, for $z\in W_h$, $\Pi_\discP z=z$. The reconstructed gradient $\nabla_\discP z$ 
is defined as the solution to
\[
\begin{aligned}
&\mathcal A\nabla_\disc z\in \bfVh\mbox{ and, for all $\bfw\in \bfVh$},\\
&\int_\O \bfw(\x)\cdot \nabla_\discP z(\x)d\x = -\int_\O z(\x)\div\bfw(\x)d\x.
\end{aligned}
\]
The existence and uniqueness of $\nabla_\discP z$ follows by applying the Riesz representation
theorem in $\bfVh$ with the inner product $(\bfw,\bfv)\mapsto \int_\O \bfw\cdot \mathcal A^{-1}\bfv d\x$.

Taking $\mathcal A(\x)=\frac{\K(\x)}{\mu(\PiDc c^{(n)}(\x))}$, the scheme
\eqref{GSpress} is exactly an $\RT0$ mixed finite element discretisation of the pressure
equation at the $n$-th time step. We notice here that $\mathcal A$, and thus the gradient discretisation
$\discP$ built above, changes with each time step; we are therefore in the context
of Remark \ref{rk:various.Dp}.

\subsubsection{Assumptions \ref{hyp:GDs}--\ref{hyp:conv.u}}\label{sec:prop.conf.mixed}

We show here that all required assumptions for Theorem \ref{th:main.convergence}
are satisfied by sequences of GDs as in Section \ref{sec:conf.mixed}.

\medskip

Under usual mesh regularity properties, Assumption \ref{hyp:GDs} follows from \cite[Chapters 8 and 9]{GDMBook16}
(note that $W_\discC\equiv 0$ and $C_\discC\le C_P$, where $C_P$
is the Poincar\'e--Wirtinger constant in $H^1(\O)$).
For the GD $\discP$ built on the $\RT0$ mixed FE, although the matrix
$\mathcal A$ changes with each time step, it always remains uniformly
bounded and coercive; the analysis in \cite{EGH:RTK.GS} thus shows that the notions of coercivity,
GD-consistency and limit-conformity as in Remark \ref{rk:various.Dp} are verified.

Thanks to \eqref{hyp:q.bounded}, the standard Darcy velocity $\discDarcyU=-\frac{\K}{\mu(\Pi_\discC c^{(n)})}\nabla_\discP p^{(n+1)}$ resulting from
the $\RT0$ discretisation of the pressure equation already satisfies Assumption \ref{hyp:hdiv},
and is therefore naturally used as tracking velocity.
Assumption \ref{hyp:conv.u.1}) is trivially satisfied since $|\discDarcyU[]|\le \Lambda_A|\nabla_\discP p|$.
Moreover, under \ref{hyp:GDs}, if $\Pi_\discC c\to c$ in $L^2(\O\times(0,T))$ as the
mesh size and time step go to $0$, then $\widetilde{\Pi}_\discC c$ also converges to $c$ in
the same space (see, e.g., end of Section \ref{sec:compactness}); thus,
if $\nabla_\discP p\to \nabla p$ in $L^2(\O\times(0,T))$, the assumption
\eqref{hyp:viscosity} on $\K/\mu$ ensures that $\discDarcyU[]=-\frac{\K}{\mu(\widetilde{\Pi}_\discC c)}\nabla_\discP p$
strongly converges in $L^2(\O\times(0,T))$ to $\darcyU=-\frac{\K}{\mu(c)}\nabla p$, which proves
\ref{hyp:conv.u.2}).

For $\discC$ coming from a conforming finite element method,
the standard nodal interpolation $\sinterp_\discC$ clearly satisfies \ref{hyp:smooth.interp}
(see \cite[Theorem 4.4.20]{brenner-scott:08}).
Finally, Assumption \ref{hyp:PiD.flow} follows from Lemma \ref{lem:flow.cont}
applied to $f=\Pi_\discC z\in H^1(\O)$, $\alpha=1$ and $r=2$.

\subsection{Finite-volume based} \label{sec:HMM}

A number of finite volume numerical schemes can be embedded in the gradient discretisation method \cite{GDMBook16}.
Here, we focus on one particular method, the Hybrid Mimetic Mixed method (HMM) \cite{dro-10-uni}, which covers in particular the
hybrid finite volume schemes \cite{egh10}, the mixed/hybrid Mimetic Finite Differences presented
for example in \cite{bre-05-con}, and the mixed finite volume method \cite{de06mixed}.
The HMM method was used in \cite{CD17,fvca8-ellam} to discretise the diffusion terms
in both the pressure and concentration equations, together with the ELLAM for the advection term.
The analysis carried out here applies to many other numerical schemes based on piecewise-constant
reconstructions, such as the VAG scheme, the MPFA-O FV method, mass-lumped FE methods
or nodal Mimetic Finite Differences \cite{GDMBook16}.

\subsubsection{Description of the HMM gradient discretisation}\label{sec:HMM.GD}

Let us first introduce a few mesh-related notations. We consider a polytopal mesh $\polyd=(\mesh,\edges,\centers)$
of $\O\subset \R^d$ as in \cite[Definition 7.2]{GDMBook16}. $\mesh$ is the set of polytopal cells (polygons in 2D, polyhedra in 3D),
$\edges$ the set of faces (edges in 2D) and $\centers$ a set of one point $\x_K$ inside each cell $K\in\mesh$.
No conformity is assumed on the mesh, which can therefore have hanging nodes, be locally refined, have non-convex cells, etc.
For $K\in\mesh$, $\edges_K$ denotes the set of faces of $K$ and, if $\edge\in\edges_K$, $D_{K,\edge}$
is the convex hull of $\edge$ and $\x_K$, $d_{K,\edge}$ is the orthogonal distance between $\x_K$
and $\edge$, $\bfn_{K,\edge}$ is the outer unit normal to $\edge$ and
$\centeredge$ is the center of mass of $\edge$ (see Figure \ref{fig:cell}).
It is assumed that each $K\in \mesh$ is star-shaped with respect to $\x_K$.

\begin{figure}[h]
	\caption{Notations inside a cell.}
	\begin{center}
		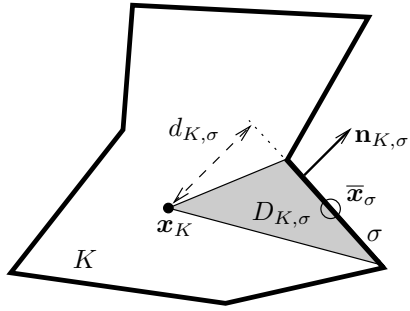
	\end{center}
	\label{fig:cell}
\end{figure}

A spatial gradient discretisation $\disc=(X_\disc,\Pi_\disc,\nabla_\disc)$ and interpolant $\ICinterp_\disc$
are then constructed by setting
\[
X_\disc=\{v=((v_K)_{K\in\mesh},(v_\edge)_{\edge\in\edges})\,:v_K\in\R\,,\;v_\edge\in\R\}
\]
(that is, there is one unknown per cell and one unknown per face),
\begin{align*}
&\forall v\in X_\disc\,,\;\forall K\in\mesh\,:\,\Pi_\disc v=v_K\mbox{ on $K$},\\
&\forall \phy\in L^2(\O)\,:\,\ICinterp_\disc \phy=((\phy_K)_{K\in\mesh},(\phy_\edge)_{\edge\in\edgescv})\in X_{\disc}\\
&\qquad\mbox{where }\phy_K=\frac{1}{|K|}\int_K \phy(\x)d\x\mbox{ and }\phy_\edge=0
\end{align*}
(note that only $\Pi_\disc \ICinterp_\disc$ is of interest -- see Definition \ref{def:propGDs} --
so the value of the edge unknowns is irrelevant for $\ICinterp_\disc$)
and
\be\label{HMM:nabla}
\begin{aligned}
&\forall v\in X_{\disc}\,,\;\forall K\in\mesh\,,\;\forall \edge\in\edgescv\,,\\
&\quad
\nabla_\disc v= \ograd_{\cv} v + \dfrac{\sqrt{d}}{\dcvedge}[v_{\edge}-v_{\cv}-\ograd_{\cv}v \cdot(\centeredge-\centercv)] \ncvedge\mbox{ on $D_{K,\edge}$}\,,\\
&\quad\mbox{ where }\ograd_{\cv}v=\frac{1}{|\cv|}\sum_{\edge\in\edgescv} |\sigma|v_{\edge}\ncvedge.
\end{aligned}
\ee
\begin{remark}
In  \eqref{HMM:nabla}, $\ograd_K v$ represents a consistent discretisation of the gradient,
in the sense that if $(v_\edge)_{\edge\in\edgescv}$ interpolate an affine mapping $A$
at the points $(\centeredge)_{\edge\in\edgescv}$, then $\ograd_K v = \nabla A$.
The second part of $\nabla_\disc v$, akin to the remainder of a discrete first order Taylor
expansion, is a stabilisation term. More general forms of stabilisation can
be chosen \cite{dro-12-gra}, but we do not describe them to simplify the presentation.
\end{remark}

An HMM scheme for \eqref{pressure} is obtained by writting \eqref{GSpress} with $\discP=\disc$ constructed above.
Such a scheme can be formulated as a finite volume scheme. Define, for $p\in X_\disc$ and $K\in\mesh$, 
the fluxes $(\mathcal F_{K,\edge}(p))_{\edge\in\edgescv}$ by
\be\label{HMM:fluxes}
\forall v\in X_\disc\,,\;\sum_{\edge\in\edgescv} \mathcal F_{K,\edge}(p)(v_K-v_\edge)=\int_K \mathcal A(\x)\nabla_\disc p(\x)
\cdot\nabla_\disc v(\x)d\x
\ee
with $\mathcal A$ a diffusion matrix. Then $p$ solves \eqref{GSpress} with $\discP=\disc$ and $\mathcal A=A(\cdot,\Pi_{\discC}c^{(n)})$
if and only if $\int_\O \Pi_\disc p=0$ and the corresponding fluxes satisfy the balance and conservativity relations, constitutive 
equations of finite volume schemes \cite{D14-FVschemes}:
\be\label{HMM:FV}
\begin{aligned}
&\forall K\in\mesh\,,\; \sum_{\edge\in\edgescv}\mathcal F_{K,\edge}(p)=\int_K (q^+-q^-)(\x)d\x,\\
&\forall \edge\mbox{ face between two cells $K$ and $L$}\,,\; \mathcal F_{K,\edge}(p)+\mathcal F_{L,\edge}(p)=0,\\
&\forall\edge\mbox{ face contained in $\partial\O$}\,,\;\mathcal F_{K,\edge}(p)=0.
\end{aligned}
\ee

\subsubsection{Assumptions \ref{hyp:GDs}--\ref{hyp:smooth.interp}}\label{sec:A.HMM}

Let us define the mesh regularity parameter
\be\label{def.reg.mesh}
\varrho_\polyd= \max_{K\in\mesh}{\rm Card}(\edgescv)+
\max_{K\in\mesh}\max_{\edge\in\edgescv} \frac{{\rm diam}(D_{K,\edge})}{{\rm inrad}(D_{K,\edge})},
\ee
where ${\rm inrad}(D_{K,\edge})$ is the radius of the largest
ball included in $D_{K,\edge}$. Under a boundedness
assumption on $\varrho_\polyd$, the basic properties \ref{hyp:GDs}
(with both $\discC$ and $\discP$ given by an HMM GD as in
Section \ref{sec:HMM.GD}) follow from the results in \cite[Chapter 12]{GDMBook16}.
The appendix of \cite{AD16} describes an interpolant $\sinterp_\disc$
and shows that it satisfies Assumption \ref{hyp:smooth.interp}.

Denoting by $Y_\mesh$ the space of piecewise constant functions on $\mesh$,
we have $\Pi_\disc(X_\disc)\subset Y_\mesh$. Recalling the definition
\eqref{def:snorm.disc} of the discrete $H^1$-semi norm on $Y_\mesh$,
\cite[Lemma 12.9 and Remark 7.6]{GDMBook16} show that $\snorm{\Pi_\disc \cdot}{\mesh}\le \beta_\disc \norm{\nabla_\disc \cdot}{L^2(\O)}$ with $\beta_\disc$ depending only on an upper bound of $\varrho_\polyd$
(this estimate is not specific to the HMM; it holds
for all currently known GDs such that $\Pi_\disc(X_\disc)\subset Y_\mesh$).
Assumption \ref{hyp:PiD.flow} is therefore a consequence of Lemma \ref{lem:pwcst.flow}, provided that the reconstructed Darcy velocity is piecewise polynomial
(which is usually the case -- see next section).

\subsubsection{Reconstructed Darcy velocity and Assumptions \ref{hyp:hdiv}--\ref{hyp:conv.u}}\label{sec:recon.darcy}

For methods like the HMM that produce piecewise-constant gradients $\nabla_{\discP}p^{(n+1)}$ and/or
piece\-wise-constant concentration $\PiDc c^{(n)}$, the natural Darcy velocity
$-\frac{\K}{\mu(\PiDc c^{(n)})}\nabla_\discP p^{(n+1)}$ does not belong to $H_\div(\O)$.
It is therefore not suitable to define the characteristics used in the ELLAM, and
another velocity must be reconstructed to be used in \eqref{Advection}. Finite-volume methods naturally produce numerical fluxes on the mesh faces, that satisfy the balance and conservativity relations
\eqref{HMM:fluxes}--\eqref{HMM:FV}. Such fluxes can be used to reconstruct
a Darcy velocity in a Raviart--Thomas space on a sub-mesh of $\mesh$.

In \cite{fvca8-ellam,CD17}, this idea is applied to the HMM method on the sub-mesh of pyramids $(D_{K,\edge})_{K\in\mesh,\,\edge\in\edgescv}$. A velocity $\discDarcyU\in H_\div(\O)$ is constructed from
the pressure unknowns such that its restriction to each diamond $D_{K,\edge}$ belongs to $\RT0$ and that,
for each cell $K\in\mesh$,
\be\label{rec.darcy}
\begin{aligned}
&\mbox{For a.e.\ $\x\in K$, }\div\discDarcyU(\x)=\frac{1}{|K|}\sum_{\edge\in\edgescv}\mathcal F_{K,\edge}(p^{(n+1)})\,,\\
&\forall \edge\in\edgescv\,,\;\forall\y\in\edge\,,\;|\edge|\discDarcyU(\y)\cdot\bfn_{K,\edge}=\mathcal F_{K,\edge}(p^{(n+1)}).
\end{aligned}
\ee
Given the flux balance equation in \eqref{HMM:FV}, this reconstruction of $\discDarcyU$
satisfies Assumption \ref{hyp:hdiv} with $M_\div=M_{q^+}+M_{q^-}$ (see \eqref{hyp:q.bounded}).

Let us now establish the estimate on $\discDarcyU[]$ stated in \ref{hyp:conv.u}.
In the following estimates, $A\lesssim B$ means that $A\le C B$ with $C$ depending
only on an upper bound of $\varrho_\polyd$, and of $\alpha_A$ and $\Lambda_A$ in \eqref{hyp:viscosity}.
Fix $K\in\mesh$. The relations \eqref{rec.darcy} boil down to a linear system for internal fluxes in $K$ --
that is, fluxes $\mathcal F_\tau$ on $(\partial D_{K,\edge}\setminus \edge)_{\edge\in\edgescv}$ --
in which the right-hand side is $(\mathcal F_{K,\edge}(p^{(n+1)}))_{\edge\in\edgescv}$. Augmenting this
system with a consistency relation or fixing the solution to be of minimal $\ell^2$ norm
(see \cite{KR03-kuznetsov-repin,CD17})
leads to a linear system $M_K (\mathcal F_\tau)_{\tau} = (\mathcal F_{K,\edge}(p^{(n+1)}))_{\edge\in\edgescv}$
with $M_K$ depending only on the number of faces of $K$, not on the geometry of this cell. Hence, $\sum_{\tau} |\mathcal F_{\tau}|^2\lesssim
\sum_{\edge\in\edgescv}|\mathcal F_{K,\edge}(p^{(n+1)})|^2$. Due to the shape regularity
assumption (which implies $|\tau|^{-1}\lesssim {\rm diam}(K)/|K|$ for any face
$\tau$ of any pyramid $D_{K,\edge}$) and by construction of $\RT0$ functions, we infer that
\begin{align}
\norm{\discDarcyU}{L^2(D_{K,\edge})}^2\lesssim{}& \sum_{\tau\subset\partial D_{K,\edge}}\frac{{\rm diam}(K)}{|\tau|}|\mathcal F_\tau|^2\nonumber\\
\lesssim{}& \frac{{\rm diam}(K)^2}{|K|}\sum_{\edge\in\edgescv}|\mathcal F_{K,\edge}(p^{(n+1)})|^2.
\label{est.L2.u.0}
\end{align}
Fix $\edge\in\edgescv$ and take, in \eqref{HMM:fluxes}, $v_\edge=1$
and $v_K=v_{\edge'}=0$ if $\edge\not=\edge'$. The definition \eqref{HMM:nabla} of $\nabla_\disc$
easily shows that $|\nabla_\disc v|\lesssim {\rm diam}(K)^{-1}$ and \eqref{HMM:fluxes} therefore
yields ${\rm diam}(K)\sum_{\edge\in\edgescv}|\mathcal F_{K,\edge}(p^{(n+1)})|\lesssim \int_K |\nabla_\disc p^{(n+1)}(\x)|d\x$.
Hence, by the Cauchy--Schwarz inequality,
\[
\frac{{\rm diam}(K)^2}{|K|}\sum_{\edge\in\edgescv}|\mathcal F_{K,\edge}(p^{(n+1)})|^2
\lesssim \norm{\nabla_\disc p^{(n+1)}}{L^2(K)}^2.
\]
Combined with \eqref{est.L2.u.0} this proves \ref{hyp:conv.u.1}).

Because of this bound, the weak convergence in \ref{hyp:conv.u.2}) follows if we can show 
that $\discDarcyU[]$ converges to $\Daru$ against any $\bvarphi\in C^\infty_c(\O\times(0,T))^d$.
To establish this convergence, we first evaluate $\discDarcyU[]-\U_\discP$,
where $\U_\discP=-\frac{\K}{\mu(\widetilde{\Pi}_\discC c)}\nabla_\disc p$.
Fix $\bxi\in\R^d$ and apply the divergence theorem between $\discDarcyU\in H_\div(K)$ and the affine map $x\mapsto \bxi\cdot(\x-\x_K)$ to write
\[
\begin{aligned}
\int_K \discDarcyU(\x)\cdot\bxi d\x={}&
\int_K \discDarcyU(\x)\cdot\nabla(\bxi\cdot(\x-\x_K)) d\x\\
={}&
\sum_{\edge\in\edgescv}\int_\edge \discDarcyU(\y)\cdot\bfn_{K,\edge} [\bxi\cdot(\y-\x_K)]\dbdr(\y)\\
&-\int_K \div\discDarcyU(\x) [\bxi\cdot(\x-\x_K)]d\x.
\end{aligned}
\]
Using then \eqref{rec.darcy} and $\frac{1}{|\edge|}\int_\edge \y\dbdr(\y)=\centeredge$ leads to
\be\label{rel.darcy.0}
\begin{aligned}
\int_K \discDarcyU(\x)\cdot\bxi d\x
={}&\sum_{\edge\in\edgescv}\mathcal F_{K,\edge}(p^{(n+1)}) \bxi\cdot (\centeredge-\x_K)\\
&-\int_K \div\discDarcyU(\x) [\bxi\cdot(\x-\x_K)]d\x.
\end{aligned}
\ee
Apply \eqref{HMM:fluxes} with $v\in X_\disc$ the interpolant of
the linear mapping $\x\mapsto\bxi\cdot\x$, that is, $v_K=\bxi\cdot\x_K$ and $v_\edge=\bxi\cdot\centeredge$.
The $\mathbb{P}_1$-exactness property of $\nabla_\disc$ \cite[Lemma 12.8]{GDMBook16} shows that $\nabla_\disc v=\bxi$ and 
\eqref{HMM:fluxes} thus gives
\[
\sum_{\edge\in\edgescv}\mathcal F_{K,\edge}(p^{(n+1)})\bxi\cdot(\centeredge-\x_K)=\int_K \U_\discP^{(n+1)}(\x)\cdot\bxi d\x.
\]
Combining
with \eqref{rel.darcy.0} and using the generality of $\bxi$ then yields
\[
\int_K \discDarcyU(\x) d\x
-\int_K \U_\discP^{(n+1)}(\x) d\x=-\int_K \div\discDarcyU(\x)(\x-\x_K)d\x.
\]
Denoting by $\Proj_\mesh:L^2(\O)^d\to L^2(\O)^d$ the orthogonal projection on the piecewise constant functions on
$\mesh$ (that is, $(P_\mesh f)_{|K}=\frac{1}{|K|}\int_Kf(\x)d\x$ for all $K\in\mesh$), the above relation gives 
\[
\norm{\Proj_\mesh (\discDarcyU-\U_\discP^{(n+1)})}{L^1(\O)}\le h_\mesh\norm{\div\discDarcyU}{L^1(\O)},
\]
where $h_\mesh=\max_{K\in\mesh}{\rm diam}(K)$ is the mesh size.
Owing to the boundedness of $\div\discDarcyU$, this shows that $\Proj_\mesh(\discDarcyU[]-\U_\discP)\to 0$ in $L^\infty(0,T;L^1(\O))$ as $h_\mesh\to 0$. 
Take now $\bvarphi\in C^\infty_c(\O\times(0,T))^d$. Using the orthogonality
property of $\Proj_\mesh$,
\begin{align}
\Big|\int_{\O\times(0,T)} {}&(\discDarcyU[]-\U_\discP)\cdot\bvarphi\Big|\nonumber\\
\le{}&
\Big|\int_{\O\times(0,T)} (\discDarcyU[]-\U_\discP)\cdot(\bvarphi-\Proj_\mesh\bvarphi)\Big|
+\Big|\int_{\O\times(0,T)} (\discDarcyU[]-\U_\discP)\cdot \Proj_\mesh\bvarphi\Big|\nonumber\\
\le{}&\norm{\discDarcyU[]-\U_\discP}{1}h_\mesh\norm{D\bvarphi}{\infty}
+\Big|\int_{\O\times(0,T)} \Proj_\mesh(\discDarcyU[]-\U_\discP)\cdot \bvarphi\Big|\nonumber\\
\le{}&\norm{\discDarcyU[]-\U_\discP}{1}h_\mesh\norm{D\bvarphi}{\infty}
+\norm{\Proj_\mesh(\discDarcyU[]-\U_\discP)}{1}\norm{\bvarphi}{\infty}.
\label{cv.track.u}
\end{align}
where $\norm{\cdot}{r}=\norm{\cdot}{L^r(\O\times(0,T))}$
and we used $\norm{\bvarphi-\Proj_\mesh\bvarphi}{\infty}\le h_\mesh \norm{D\bvarphi}{\infty}$.
The strong convergence of $\Pi_{\discC}c$ ensures the strong convergence
of $\widetilde{\Pi}_\discC c$ (see end of Section \ref{sec:compactness});
hence, the strong convergences assumed
in \ref{hyp:conv.u} show that $\U_\discP\to \Daru=-\frac{\K}{\mu(c)}\nabla p$ in $L^2(\O\times(0,T))^d$. Since 
the right-hand side of \eqref{cv.track.u} tends to $0$ as $h_\mesh\to 0$,
this concludes the proof that $\discDarcyU[]\to \Daru$ weakly in $L^2(\O\times(0,T))^d$ as 
the mesh size and time step tend to $0$.

\section{Properties of the Flow}\label{sec:flow}

A few properties on the solution of the characteristic equation \eqref{charac} are established here.
To simplify the notations, we set $\discDarcyU=\Vel$.
Hence, for $\x\in\O$, $t\mapsto F_t(\x)$ solves
\begin{equation} \label{charac.appendix}
\dfrac{dF_{t}(\x)}{dt} =\dfrac{\Vel(F_{t}(\x))}{\phi(F_t(\x))}\quad\mbox{ for $t\in [-T,T]$}, \qquad F_0(\x)= \x.
\end{equation}
Associated with the flow equation \eqref{charac.appendix} is the advection equation
\be\label{adv.w}
\phi\partial_t w + \Vel\cdot\nabla w=0.
\ee
A function $w$ is a solution to such an equation if it satisfies, for all $s,t\in [-T,T]$ such that
$s-t\in [-T,T]$ and for a.e.\ $\x\in\O$, $w(\x,t)=w(F_{s-t}(\x),s)$.

These flow and advection equations will be studied under the assumptions \eqref{hyp:phi} on $\phi$ and
\ref{hyp:hdiv} on $\Vel$. Upon considering a common sub-mesh of the meshes considered in 
these assumptions
there is no loss in generality in assuming that the meshes for $\phi$ and $\Vel$ are the same.
In other words, our leading assumption here is: there is a mesh $\mesh$ (that is, a partition of
$\O$ into polygonal/polyhedral cells) such that
\be\label{assump.phiV}
\begin{aligned}
	&\mbox{$\phi$ is piecewise smooth on $\mesh$ and $\phi_*\le \phi\le \phi^*$,}\\
	&\mbox{$\Vel\in H_\div(\O)$ is piecewise polynomial on $\mesh$,}\\
	&\mbox{$|\div\Vel|\le \Gamma_\div$ and $\Vel\cdot\bfn=0$ on $\partial\O$.}
\end{aligned}
\ee

\begin{lemma}[The flow is well-defined]\label{lem:def.flow}
Under Assumption \eqref{assump.phiV}, there exists a closed set $\mathcal C\subset\O$
with zero Lebesgue measure such that, for any $\x\in \O\backslash\mathcal C$, there is
a unique Lipschitz-continuous map $t\in [-T,T]\mapsto F_t(\x)\in\O\backslash\mathcal C$
that satisfies \eqref{charac.appendix}
(except at an at most countable number of times for the ODE). Moreover, $F_t$ has classical flows properties:
for all $t\in [-T,T]$, $F_t:\O\backslash \mathcal C\to \O\backslash\mathcal C$ is
a  locally Lipschitz-continuous homeomorphism (which can thus be used for changes of variables in integrals), and
$F_{t+s}=F_t\circ F_s$ for all $s,t\in [-T,T]$ such that $s+t\in [-T,T]$.
\end{lemma}

\begin{proof}
By smoothness of $\Vel$ and $\phi$ in each cell, the flow $t\mapsto F_t(\x)$ of $\Vel/\phi$ can clearly be
defined until it reaches a cell boundary. Assume that it reaches at a time $t=t_\edge$ a cell boundary
at a point $\y$ that is not a vertex or on an edge of the cell (we use here the 3D nomenclature), that is,
$\y$ is in the relative interior of a face $\edge$. Denote by $H_1$ and $H_2$ the two half-spaces on
each side of $\edge$, and by $\bfn_\edge$ the normal to $\edge$ from $H_1$ to $H_2$.
Since $\Vel\in H_\div(\O)$, $\Vel\cdot\bfn_\edge$ is continuous across $\edge$.
The function $\phi$ being positive, it means that the \emph{sign}, if not the value,
of $(\Vel/\phi)\cdot\bfn_\edge$ is continuous across $\edge$.
Assuming for example that $(\Vel/\phi)_{|H_1}(\y)\cdot\bfn_\edge>0$,
then the flow arrives at $\y$ from $H_1$ and, $(\Vel/\phi)_{|H_2}(\y)\cdot\bfn_\edge$
being also strictly positive, $t\mapsto F_t(\x)$ can be restarted from $(t_\edge,\y)$
by considering $(\Vel/\phi)_{|H_2}$ (which drives the flow into $H_2$).
Note that the $H_\div$-property of $\Vel$ is essential here to ensure that the flow can
indeed be continued into $H_2$, and that the values of $\Vel/\phi$ at $\y$ from $H_1$ and $H_2$ do not
simultaneously drive the flow in the other domain, thus freezing it at $\y$.

Following this process, the flow can be continued as long as it does not cross (or starts
from) a vertex/edge or, for a face $\edge$, the set $Z_\edge=\{\y\in\edge\;:\;\Vel(\y)\cdot\bfn_\edge=0\}$.
Let $\mathcal C$ be the set consisting of all $\x\in\O$ whose flow arrive (or starts from) at a vertex/edge,
or one of the sets $Z_\edge$. The set $\mathcal C$ can be obtained by tracing back on $[-T,T]$, following the process
above, the vertices, edges or sets $Z_\edge$ (until the flow can no longer be constructed,
that is, the tracing-back process arrives on a vertex, edge or a set $Z_{\edge'}$). Since each
such set is closed, $\mathcal C$ is closed.
Moreover, vertices and edges have dimension $d-2$ or less, and are therefore traced-back by the flow
into sets of zero $d$-dimensional measure. Consider now a set $Z_\edge$. Since $\Vel\cdot\bfn_\edge$
is a polynomial, either $Z_\edge=\edge$ or $Z_\edge$ has dimension $d-2$ or less. In the latter
case, as for vertices/edges, its traced-back set has zero $d$-dimensional measure.
If $Z_\edge=\edge$, then $\Vel$ is parallel to $\edge$ (whatever the side we consider for
the values of $\Vel$) and the traced-back region of $Z_\edge$ is contained in $\overline{\edge}$,
which has zero $d$-dimensional measure. Hence, $\mathcal C$ has zero $d$-dimensional measure.
This reasoning also shows that the flow never crosses the boundary of $\O$, since $\Vel\cdot\bfn=0$
on $\partial\O$.

This construction ensures that, for all $\x\not\in\mathcal C$, the flow $t\mapsto F_t(\x)\in \O\backslash \mathcal C$ is well-defined on $[-T,T]$, satisfies
the ODEs except at a countable number of points (where it intersects faces), is Lipschitz-continuous
(since it is globally continuous and Lipschitz inside each cell, with a Lipschitz constant bounded by $\norm{\Vel}{L^\infty(\O)}/\phi_*$), and satisfies the flow property $F_{t+s}=F_t\circ F_s$. To see that it is locally Lipschitz on $\O\backslash \mathcal C$ with respect to its base point $\x$,
we simply have to notice that for $\x\not\in\mathcal C$, by construction of $\mathcal C$,
there is a ball $B(\x,\theta)$ centered at $\x$ such that, for any $\y\in B(\x,\theta)$,
the flow $t\mapsto F_t(\y)$ travels into the same cells and crosses the same faces as
$t\mapsto F_t(\x)$. Since, in each cell, the flow is Lipschitz-continuous w.r.t.\ its base
point with a uniform Lipschitz constant (because $\Vel$ and $\phi$ are smooth in each cell, with bounded
derivatives), gluing the Lipschitz estimate thanks to the flow property we can check that $\y\mapsto F_t(\y)$
is Lipschitz continuous on $B(\x,\theta)$. Note that because the open set $\O\backslash \mathcal C$
can be disconnected, this does not prove a global Lipschitz property of the flow.

The homoeomorphism property follows from the flow property which shows that, on $\O\backslash \mathcal C$,
$F_t\circ F_{-t}=F_0={\rm Id}$.
\end{proof}

Let us now establish some relations and estimates on this flow.

\begin{lemma}[Estimates on the flow] Under Assumptions \eqref{assump.phiV}, for a.e.\ $\x\in\O$ and all $s\in [-T,T]$,
	denoting by $JF_t$ the Jacobian determinant of $F_t$,
	\begin{equation}\label{JacobianInt}
	\int_0^s |JF_{t}(\x)| (\divg \Vel)\circ F_{t}(\x) dt = \phi(F_{s}(\x))|JF_{s}(\x)|- \phi(\x)
	\end{equation}
	and
	\be\label{JF.bd}
	|JF_s(\x)|\le C_1(s):=\dfrac{\phi^*}{\phi_*} \exp\left(\frac{\Gamma_\div}{\phi_*}|s|\right).
	\ee
	Moreover, let $w\ge 0$ be a solution of \eqref{adv.w}.
	Then, for all $s,t\in [-T,T]$ such that $s-t\in [-T,T]$,
	\begin{equation}\label{Pv2Ineq}
	\begin{aligned}
	\int_\O \phi(\x) w(\x,t-s) d\x&\leq \left(1+\frac{\Gamma_\div C_1(T)}{\phi_*}|s|\right)
	\int_{\O} \phi(\x) w(\x,t)  d\x
	\end{aligned}
	\end{equation}
	and
	\begin{equation}\label{v2Ineq}
	\int_\O w(\x,t-s)d\x \leq \frac{C_1(T)}{\phi_*} \int_{\O} \phi(\x) w(\x,t) d\x.
	\end{equation}
\end{lemma}

\begin{proof}
	
	\textbf{Step 1}: we establish the following generalised Liouville formula:
	for any measurable set $A\subset \O$,
	\begin{equation}\label{gen.liouville}
	\dfrac{d}{dt}\int_{F_{t}(A)} \phi(\y)d\y = \int_{F_{t}(A)} \divg \Vel(\y)d\y, 
	\end{equation}
	where the time derivative $\dfrac{d}{dt}$ is taken in the sense of distributions (this also shows
	that the function $t\mapsto  \int_{F_{t}(A)} \phi(\y)d\y$ belongs to $W^{1,1}(-T,T)$).
	
	Let $v_0\in C^\infty_c(\O)$ and set $v(\x,t)=v_0(F_{-t}(\x))$. Then $v$ is Lipschitz-continuous with respect
	to $t$ and, by the flow property, $v(\x,t)=v(F_{s-t}(\x),s)$. Hence, 
	\[
	\partial_t v(\x,t)=\nabla v(F_{s-t}(\x),s)\cdot\frac{d}{dt}(F_{s-t}(\x))=-\nabla v(F_{s-t}(\x),s)\cdot \frac{\Vel(F_{s-t}(\x))}{\phi(F_{s-t}(\x))}.
	\]
	Given the piecewise regularity assumptions on $\Vel$ and $\phi$, for a.e.\ $\x\in\O$ we can let $s\to t$ in the
	above relation to find $\partial_t v(\x,t)=-\nabla v(\x,t)\cdot \frac{\Vel(\x)}{\phi(\x)}$. Hence,
	since $\Vel\in H_\div(\O)$ with $\Vel\cdot\bfn=0$ on $\partial\O$,
	\begin{multline*}
	\frac{d}{dt}\int_\O \phi(\x)v(\x,t)d\x = \int_\O \phi(\x)\partial_t v(\x,t)d\x\\
	=-\int_\O \nabla v(\x,t)\cdot\Vel(\x)d\x = \int_\O v(\x,t)\div\Vel(\x)d\x.
	\end{multline*}
	Let us now take a sequence $(v_0^{(n)})_{n\in\N}$ in $C^\infty_c(\O)$
	that converges a.e.\ on $\O$ to the characteristic function $\mathbf{1}_A$ of $A$, and such that
	$0\le v_0^{(n)}\le 1$. The relation above yields
	\be\label{for.liouville}
	\frac{d}{dt}\int_\O \phi(\x)v_0^{(n)}(F_{-t}(\x))d\x = \int_\O v_0^{(n)}(F_{-t}(\x))\div\Vel(\x)d\x.
	\ee
	As $n\to\infty$, the right-hand side converges (by dominated convergence) to
	\[
	\int_\O \mathbf{1}_A(F_{-t}(\x))\div\Vel(\x)d\x=\int_{F_{t}(A)} \div\Vel(\x)d\x.
	\]
	The sequence of mappings $t\mapsto \int_\O \phi(\x)v_0^{(n)}(F_{-t}(\x))d\x$
	converge pointwise to 
	\[
	t\mapsto \int_\O \phi(\x)\mathbf{1}_A(F_{-t}(\x))d\x=\int_{F_t(A)} \phi(\x)d\x,
	\]
	while remaining bounded. Hence, they converge weakly-$*$ in $L^\infty(-T,T)$.
	We can therefore pass to the distributional limit in \eqref{for.liouville} to see
	that \eqref{gen.liouville} holds.
	
	\medskip
	
	\textbf{Step 2}: estimates on $JF_t$.
	
	Set $A=B(\x,r)$ a ball of center $\x$ and radius $r$ contained in $\O$.
	Integrating \eqref{gen.liouville} with respect to time from $0$ to $s$ and
	using a change of variables $\y=F_{-t}(\x)$, we obtain
	\begin{multline*}
	\int_{B(\x, r)} \phi(F_{s}(\y))|JF_{s}(\y)|d\y - \int_{B(\x, r)} \phi(\y)d\y \\
	=
	\int_0^s \int_{B(\x, r)} |JF_{t}(\y)| (\divg \Vel) \circ F_{t}(\y) dtd\y.
	\end{multline*}
	Dividing by the measure of $B(\x,r)$ and taking the limit as $r \rightarrow 0$, we obtain
	\eqref{JacobianInt} for a.e.\ $\x\in\O$, due to the piecewise smoothness of $\Vel$ and $\phi$.
	
	Assume to simplify the writing that $s\ge 0$ and use the assumption on $\div\Vel$ to deduce from
	\eqref{JacobianInt} that $\phi(F_{s}(\x)) |JF_{s}(\x)| -\phi(\x) \leq \Gamma_\div \int_0^s |JF_{t}(\x)| dt$,
	and thus that
	\[
	|JF_{s}(\x)| \leq \dfrac{\phi^*}{\phi_*}+\dfrac{\Gamma_\div}{\phi_*}\int_0^s |JF_{t}(\x)| dt.
	\]
	Use then Gronwall's inequality to obtain \eqref{JF.bd}.
	
	\medskip
	
	\textbf{Step 3}: Estimates on $w$.
	
	We recall that $w(\x,t-s)=w(F_{s}(\x),t)$. Hence, a change of variables and
	\eqref{JacobianInt} yield
	\begin{align*}
	\int_\O \phi(\x) w(\x,t-s)d\x ={}& \int_{\O}\phi(\x)  w(F_{s}(\x),t)d\x
	= \int_{\O} w(\y,t)\phi(F_{-s}(\y)) |JF_{-s}(\y)|d\y \\
	={}& \int_{\O} w(\y,t)\left( \phi(\y) +\int_0^{-s} |JF_{\rho}(\y)| (\divg \Vel)\circ F_{\rho}(\y) d\rho \right) d\y.
	\end{align*}
	Estimate \eqref{Pv2Ineq} follows by writing, thanks to \eqref{JF.bd}, for a.e.\ $\y\in\O$,
	\[
	\left|\int_0^{-s} |JF_{\rho}(\y)| (\divg \Vel)\circ F_{\rho}(\y) d\rho \right|\le \Gamma_\div C_1(T)|s|
	\le \frac{\Gamma_\div C_1(T)}{\phi_*}|s|\phi(\y).
	\]
	To establish \eqref{v2Ineq}, we simply write, still using a change of variables,
	\[
	\int_\O  w(\x,t-s)d\x = \int_{\O} w(F_{s}(\x),t)d\x
	= \int_{\O} w(\y,t) |JF_{-s}(\y)|d\y
	\]
	and we use \eqref{JF.bd} and $\phi\ge \phi_*$ to conclude.
\end{proof}

The following lemma is used to prove that conforming discretisations satisfy
Assumption \ref{hyp:PiD.flow} (see Section \ref{sec:prop.conf.mixed}), and to
establish convergence properties, as the time
step tends to $0$, of functions transported by the flow (see Lemma \ref{lem:conv.func.trans}).

\begin{lemma}[Translation estimate for Sobolev functions]\label{lem:flow.cont}
	Under Assumption \eqref{assump.phiV}, let $F_t$ be the flow defined by \eqref{charac.appendix},
	and let $r,\alpha\in [1,\infty]$ be such that $\frac{1}{\alpha}=\frac{1}{2}+\frac{1}{r}$. Then, for any $f\in W^{1,r}(\O)$ and $s\in [-T,T]$,
	\[
	\norm{f(F_s)-f}{L^\alpha(\O)}\le \frac{C_1(T)^{1/\alpha}}{\phi_*} |s|\norm{\Vel}{L^2(\O)}\norm{\nabla f}{L^r(\O)},
	\]
	where $C_1(T)=\frac{\phi^*}{\phi_*}\exp(\frac{\Gamma_\div T}{\phi_*})$ as in \eqref{JF.bd}.
\end{lemma}

\begin{proof} By density it suffices to prove the estimate for $f\in C^1(\overline{\O})$ (in
the case $r=\infty$, we first establish it for $r<\infty$ and corresponding $\alpha_r$, using the density of smooth functions
in $W^{1,r}$, and then
let $r\to\infty$). For a.e.\ $\x\in\O$,
	\begin{align*}
	f(F_s(\x))-f(\x)=\int_0^s \frac{d}{dt}f(F_t(\x))dt={}&\int_0^s \nabla f(F_t(\x))\cdot\frac{dF_t(\x)}{dt} dt\\
	={}&\int_0^s \nabla f(F_t(\x))\cdot\frac{\Vel(F_t(\x))}{\phi(F_t(\x))} dt
	\end{align*}
	Take the absolute value, the power $\alpha$ (using Jensen's inequality) and integrate over $\O$.
	Using $\phi\ge \phi_*$ and applying a change of variables $\y=F_{t}(\x)$ along with \eqref{JF.bd}, this leads to
	\begin{align*}
	\int_\O |f(F_s(\x))-f(\x)|^\alpha d\x\le{}&\frac{|s|^{\alpha-1}}{\phi_*^\alpha}\int_\O \int_{[0,s]} |\nabla f(F_t(\x))|^\alpha|\Vel(F_t(\x))|^\alpha dt d\x\\
	\le{}&\frac{|s|^{\alpha-1}}{\phi_*^\alpha}\int_{[0,s]} \left(\int_{\O} |\nabla f(F_t(\x))|^\alpha|\Vel(F_t(\x))|^\alpha d\x\right)dt\\
	\le{}&\frac{C_1(T)|s|^\alpha}{\phi_*^\alpha}\int_\O |\nabla f(\y)|^\alpha|\Vel(\y)|^\alpha d\y.
	\end{align*}
	The proof is complete by applying H\"older's estimate with exponents $r/\alpha$ and $2/\alpha$, and by
	taking the power $1/\alpha$ of the resulting inequality.
\end{proof}

We now want to establish a similar result but for piecewise-constant
functions. This will be useful to establish that discretisations based on piecewise-constant
approximations, such as most FV methods, satisfy Assumption \ref{hyp:PiD.flow}.
Before stating this lemma, we need a preliminary result.

\begin{lemma}[Volume covered by a face transported by the flow] \label{edgeVolume}
	Under Assumption \eqref{assump.phiV}, let $F_t$ be the flow defined by \eqref{charac.appendix}.
	Let $\edge$ be a face of the mesh over which $\Vel$ and $\phi$ are piecewise smooth.
	Let $V_t=|F_{[0,t]}(\sigma)|$ be the volume of the region covered by $\sigma$ when transported over $[0,t]$ by
	the flow, that is, $V_t=|\{F_s(\y)\,:\,s\in [0,t]\,,\;\y\in\edge\}|$.
	Then
	\be\label{est.Vt}
	\forall t\in [-T,T]\,,\;V_t \leq \frac{C_1(T)}{\phi_*}|t| \int_\edge |\Vel(\y) \cdot \bfn_{\edge}|\dbdr(\y),
	\ee
	where $C_1(T)$ is given by \eqref{JF.bd} and $\bfn_{\edge}$ is a normal to $\edge$.
\end{lemma}

\begin{proof}
	Notice first that since $\Vel\in H_\div(\O)$, the normal components of $\Vel$ across the faces of
	the mesh are continuous, and thus $|\Vel(\y) \cdot \bfn_{\edge}|$ is independent of the side of $\edge$
chosen to compute $\Vel$. Without loss of generality, we assume $t\ge 0$.

	If the face $\edge$ is such that $Z_\edge:=\{\y\in\edge\;:\;\Vel(\y)\cdot\bfn_{\edge}=0\}=\edge$,
	then even though $\edge\subset\mathcal C$ (see Lemma \ref{lem:def.flow} and its proof),
	we clearly have $V_t=0$ since each point on the face is transported inside the face to one of its
	vertex/edge, which are $(d-2)$-dimensional objects then transported by the flow onto null sets in $\O$
 (whatever side of $\edge$ chosen to compute $\Vel$ and $\phi$). Hence, \eqref{est.Vt} holds
for such faces.

	Let us now assume that $Z_\edge\not=\edge$. Then, since $\Vel\cdot\bfn_\edge$ is
polynomial, $Z_\edge$ is a negligible set in $\edge$ for the $(d-1)$-dimensional measure and
$F_t(\y)$ is defined for all $\y\in\edge\backslash Z_\edge$.
	Since $F_{[0,t+h]}(\edge)=F_{[0,t]}(\edge)\cup F_{(t,t+h]}(\edge)$, the flow property,
	a change of variables and \eqref{JF.bd} yield
	\be
	\begin{aligned}\label{VolumeChange}
		V_{t+h}-V_{t} =|F_{(t,t+h]}(\edge)|={}&|F_t(F_{(0,h]}(\edge))| 	\\
		={}&\int_{F_{(0,h]}(\edge)}|JF_t(\y)|d\y
		\leq C_1(T)|F_{(0,h]}(\edge)|. 
	\end{aligned}
	\ee
	Choose an orthonormal basis of $\R^d$ such that $\edge\subset \{0\}\times \R^{d-1}$
	and $\bfn_\edge=(1,0,\ldots,0)$, and define $G:\R\times\edge\to \R^d$ by $G(t,\y)=F_t(\y)$.
	Using the area formula \cite[Theorem 1]{EG92} we have
	\begin{multline}
	|F_{(0,h]}(\edge)|= \int_{\R^d} \mathbf{1}_{G((0,h]\times\edge))}(\x)d\x
	\le \int_{\R^d} {\rm Card}\left[((0,h]\times\edge)\cap G^{-1}(\{\x\})\right](\x)d\x\\
	=\int_{(0,h]\times\edge}|JG(t,\y)|dt\dbdr(\y)=
	\int_0^h \left(\int_\edge |JG(t,\y)|\dbdr(\y)\right)dt
	\label{est.Vt.0}
	\end{multline}
	where $JG$ is the Jacobian determinant of $G$. Given the choice of basis in the range of $G$,
	\begin{align*}
	JG(t,\y) ={}& \det \left[ \begin{array}{cccc} \frac{\partial G}{\partial t}(t,\y)&\frac{\partial G}{\partial y_1}(t,\y)
	&\cdots&\frac{\partial G}{\partial y_{d-1}}(t,\y)\end{array}\right]\\
	={}&\det \left[ \begin{array}{cccc} \frac{dF_t}{dt}(\y)&\frac{\partial F_t}{\partial y_1}(\y)
	&\cdots&\frac{\partial F_t}{\partial y_{d-1}}(\y)\end{array}\right]\\
	={}&\det \left[ \begin{array}{cccc} \frac{\Vel(F_t(\y))}{\phi(F_t(\y))}&\frac{\partial F_t}{\partial y_1}(\y)
	&\cdots&\frac{\partial F_t}{\partial y_{d-1}}(\y)\end{array}\right].
	\end{align*}
	For a fixed $\y\in\edge\backslash Z_\edge$ and for small $t$ the flow $F_t(\y)$ occurs in a region where $\Vel$ and $\phi$ (and thus $F_t$) are smooth -- namely, the side of $\edge$ determined by the sign of $\Vel(\y)\cdot\bfn_\edge$. Hence, since $F_0={\rm Id}$, denoting by $(\Vel_1,\ldots,\Vel_d)$ the components of $\Vel$ in the
	chosen basis and recalling that $\bfn_\edge=(1,0,\ldots,0)$,
	\begin{align}
	\lim_{t\to 0}JG(t,\y)={}& \det \left[ \begin{array}{cccc} \frac{\Vel(\y)}{\phi(\y)}&\frac{\partial F_0}{\partial y_1}(\y)
	&\cdots&\frac{\partial F_0}{\partial y_{d-1}}(\y)\end{array}\right]\nonumber\\
	={}&\det \left[ \begin{array}{ccccc} 
	\frac{\Vel_1(\y)}{\phi(\y)}&0&\cdots&\cdots&0\\
	\vdots&1&0&\cdots&0\\
	\vdots&0&\ddots&\ddots&\vdots\\
	\vdots&\vdots&\ddots&\ddots&0\\
	\frac{\Vel_d(\y)}{\phi(\y)}&0&\cdots&0&1
	\end{array}\right]=\frac{\Vel_1(\y)}{\phi(\y)}=\frac{\Vel(\y)\cdot\bfn_\edge}{\phi(\y)}.
	\label{lim.Gt}	
	\end{align}
	Here, the value of $\phi$ is of course considered on the side of $\edge$ into which
	$F_t(\y)$ flows for small $t>0$ (as already noticed, the value of $\Vel(\y)\cdot\bfn_\edge$
	does not depend on the considered side). Recalling that \eqref{lim.Gt} holds for $\y\in\edge\backslash
	Z_\edge$ and that $Z_\edge$ has zero $(d-1)$-dimensional measure,
	the dominated convergence theorem thus shows that
	\[
	\int_\edge |JG(t,\y)|\dbdr(\y) \to \int_\edge \frac{|\Vel(\y)\cdot\bfn_\edge|}{\phi(\y)}\dbdr(\y)\mbox{ as $t\to 0$}.
	\]
	Dividing \eqref{est.Vt.0} by $h$, letting $h\to 0$, and plugging the result in \eqref{VolumeChange} we infer that
	\[
	\frac{dV_{t}}{dt}\le \frac{C_1(T)}{\phi_*}\int_\edge |\Vel(\y)\cdot\bfn_\edge|\dbdr(\y).
	\]
	The mapping $t\mapsto V_t$ is a non-decreasing function, so its derivative in the sense of distributions
	always exists as a positive measure; the relation above shows that this derivative is actually
	a bounded function, and thus that $t\mapsto V_t$ is Lipschitz-continuous.
	Integrating this relation and using $V_0=0$ leads to \eqref{est.Vt}. \end{proof}

We can now state a result that mimics Lemma \ref{lem:flow.cont} but for piecewise-constant
functions. This result is used in Section \ref{sec:A.HMM} to prove that HMM schemes, among
others, satisfy \ref{hyp:PiD.flow}.

\begin{lemma}[Translation estimate for piecewise-constant functions]\label{lem:pwcst.flow} 
	Let $\polyd$ be a po\-lytopal mesh and $Y_\mesh$ be the set of piecewise-constant
	functions on $\mesh$. Define the discrete $H^1$-semi norm on $Y_\mesh$ by
	\be\label{def:snorm.disc}
	\forall f\in Y_\mesh\,,\;\snorm{f}{\polyd}=\left(\sum_{\edge\in\edgesint}|\edge|d_{\edge}\left|\frac{f_K-f_L}{d_{\edge}}\right|^2\right)^{1/2},
	\ee
	where $f_K$ is the constant value of $f$ on $K\in\mesh$,
	$\edgesint$ is the set of internal faces (that is, $\edge\in\edges$ such that $\edge\subset \O$),
	$K$ and $L$ are the two cells on each side of $\edge$, and $d_\edge=d_{K,\edge}+d_{L,\edge}$ (see Figure \ref{fig:cell}).
	Assume that $(\phi,\Vel)$ satisfy \eqref{assump.phiV} on the sub-mesh made of
	$(D_{K,\edge})_{K\in\mesh,\,\edge\in\edgescv}$ and let $k$ be the maximal polynomial degree of $\Vel$.
	
	Then, if $\varrho\ge\varrho_\polyd$ (defined by \eqref{def.reg.mesh}), there exists $R$ depending only on $k$, $d$ and $\varrho$ such that, for all $s\in [-T,T]$,
	\[
	\forall f\in Y_\mesh\,,\; \norm{f(F_s)-f}{L^1(\O)}
	\le R\frac{C_1(T)}{\phi_*} |s|\norm{\Vel}{L^2(\O)}\snorm{f}{\polyd}
	\]
	where $C_1(T)=\frac{\phi^*}{\phi_*}\exp(\frac{\Gamma_\div T}{\phi_*})$ as in \eqref{JF.bd}.
\end{lemma}

\begin{proof}
	We start by writing $f(F_s(\x))-f(\x)$ as the sum of the jumps
	of $f$ along the curve $(F_t(\x))_{t\in [0,s]}=:F_{[0,s]}(\x)$.
	For $\edge\in\edgesint$, letting $\chi_\edge(\x)=1$ if $\edge\cap F_{[0,s]}(\x)\not=\emptyset$
	and $\chi_\edge(\x)=0$ otherwise, this leads to
	\be\label{pwconst.flow.0}
	|f(F_s(\x))-f(\x)|\le\sum_{\edge\in\edgesint}\chi_\edge(\x)|f_K-f_L|.
	\ee
	Notice that $\edge\cap F_{[0,s]}(\x)\not=\emptyset$ if and only if
	$F_{[-s,0]}(\edge)\cap \{\x\}\not=\emptyset$, that is, $\x$ belongs to
	the region covered by $\edge$ transported by the flow over $[-s,0]$.
	Lemma \ref{edgeVolume} gives
	\[
	\int_\O \chi_\edge(\x)d\x\le \frac{C_1(T)}{\phi_*}|s| \int_\edge |\Vel(\y)\cdot\bfn_\edge|\dbdr(\y)
	\]
	where $\bfn_\edge$ is a unit normal to $\edge$. Hence, letting $C=\frac{C_1(T)}{\phi_*}$
	and using the Cauchy--Schwarz inequality (on the combined sum and integral terms),
	\begin{align}
	\int_\O|{}&f(F_s(\x))-f(\x)|d\x\nonumber\\
	\le{}& C|s| \sum_{\edge \in \edgesint} \int_\edge |\Vel(\y) \cdot \bfn_{\edge}|\,|f_{K}-f_{L}| \dbdr(\y) \nonumber\\
	={}& C|s| \sum_{\edge \in \edgesint} \int_\edge \sqrt{d_\edge}|\Vel(\y) \cdot \bfn_{\edge}|\,\frac{1}{\sqrt{d_\edge}}|f_{K}-f_{L}| \dbdr(\y) \nonumber\\
	\le{}& C|s| \left(\sum_{\edge \in \edgesint} \int_\edge d_\edge|\Vel(\y) \cdot \bfn_{\edge}|^2\dbdr(\y)\right)^{1/2}
	\left(\sum_{\edge \in \edgesint} \int_\edge \frac{1}{d_\edge}|f_{K}-f_{L}|^2 \dbdr(\y) \right)^{1/2}\nonumber\\
	={}& C|s| \left(\sum_{\edge \in \edgesint} d_\edge\int_\edge |\Vel(\y) \cdot \bfn_{\edge}|^2\dbdr(\y)\right)^{1/2}
	\snorm{f}{\polyd}.
	\label{pwcst.flow.1}
	\end{align}
	Since $\Vel$ is polynomial on each $D_{K,\edge}$, we can use the discrete trace inequality
	of \cite[Lemma 1.46]{DGbook-DE12} to find $R$ depending only on $k$, $d$ and $\varrho$
	such that
	\[
	\forall K\in\mesh\,,\;\forall \edge\in\edgescv\,,\; {\rm diam}(D_{K,\edge})
	\int_\edge |\Vel(\y) \cdot \bfn_{\edge}|^2\dbdr(\y)\le R^2 \int_{D_{K,\edge}}|\Vel(\x)|^2d\x.
	\]
	Noticing that $d_{K,\edge}\le {\rm diam}(D_{K,\edge})$, we infer
	\[
	d_\edge\int_\edge |\Vel(\y) \cdot \bfn_{\edge}|^2\dbdr(\y)\le R^2 \int_{D_{K,\edge}\cup D_{L,\edge}}|\Vel(\x)|^2d\x.
	\]
	The proof of the lemma is completed by plugging this estimate into \eqref{pwcst.flow.1}.
\end{proof}

\begin{remark}[Estimate in $L^\alpha$ norm?]
	A natural question would be the extension of Lemma \ref{lem:pwcst.flow} to estimate
	the $L^\alpha$ norm of $f(F_s)-f$, as in Lemma \ref{lem:flow.cont},
	by using the discrete $W^{1,r}$-semi norm $\snorm{f}{\polyd,r}$ of $f$ obtained by replacing $2$ with $r$
	in \eqref{def:snorm.disc}. Considering for example
	the simple case of a constant unit velocity $\Vel=\Vel_0$ (and forgetting about boundary conditions for
	simplification), this would amount to estimating $\norm{f(\cdot+s\Vel_0)-f}{L^\alpha(\O)}$
	in terms of $|s|\snorm{f}{\polyd,r}$. For meshes admissible for the TPFA finite volume
	scheme, such an estimate is known with $\alpha=r=2$ and $|s|$ replaced by
	$\sqrt{|s|(|s|+\max_{K\in\mesh}{\rm diam}(K))}$ \cite[Lemma 3.3]{EGH00}.
	For general meshes, however, no similar estimate seems to be attainable if $\alpha>1$.
\end{remark}

The next lemma is instrumental in passing to the limit in the reaction and advection terms of
the GDM--ELLAM scheme. Let us first introduce some notations. 
Given time steps $0=t^{(0)}<t^{(1)} <\ldots <t^{(N)}=T$ and velocities
$\Vel=(\Vel^n)_{n=1,\ldots,N}$ that satisfy \eqref{assump.phiV}, we identify $\Vel$
with the global velocity $\O\times(0,T)\to \R^d$ given by
$\Vel(\cdot,t)=\Vel^{(n+1)}$ for all $t\in (t^{(n)},t^{(n+1)}]$ and all $n=0,\ldots,N-1$.
Define $\trans_\Vel$ and $\transh_\Vel$ as the linear mappings $L^2(\O\times(0,T))\to L^2(\O\times(0,T))$
such that, for $ \psi\in L^2(\O\times(0,T))$,
\be\label{def.taum}
\begin{aligned}
	&\mbox{for a.e.\ $\x\in\O$,
		for all $t\in(t^{(n)},t^{(n+1)})$ and $n=0,\ldots,N-1$,}\\
	&\trans_{\Vel}\psi(\x,t) =  \psi\left(F^{(n+1)}_{t^{(n+1)}-t^{(n)}}(\x),t\right)
	\quad\mbox{ and }\quad\transh_{\Vel}\psi(\x,t) =  \psi\left(F^{(n+1)}_{t^{(n+1)}-t}(\x),t\right)
\end{aligned}
\ee
where $F^{(n+1)}_{t}$ is defined by \eqref{charac.appendix} for the velocity $\Vel^{(n+1)}$.
The difference between $\trans_\Vel$ and $\transh_\Vel$ is the time at which this flow is considered.

\begin{lemma}[Convergence of functions transported by the flow]\label{lem:conv.func.trans}
	Let $\phi$ satisfy \eqref{hyp:phi} and, for each $m\in\N$, take $0=t_m^{(0)}<t_m^{(1)} <\ldots <t_m^{(N_m)}=T$ time steps 
	and $\Vel_m=(\Vel_m^n)_{n=1,\ldots,N_m}$ that satisfy \eqref{assump.phiV} with $\Gamma_\div$
	not depending on $m$. Assume that $\dt_m:=\max_{n=0,\ldots,N_m-1} (t_m^{(n+1)}-t_m^{(n)}) \rightarrow 0$
	as $m\rightarrow \infty$ and that $(\Vel_m)_{m\in\N}$ is bounded in $L^2(\O\times(0,T))$.
	Then $\trans_{\Vel_m}$ and $\transh_{\Vel_m}$ satisfy the following properties.
	\begin{enumerate}
		\item There is $C$ not depending on $m$ such that, for $\psi\in L^2(\O\times(0,T))$,
		\be\label{est.L2.trans}
		\norm{\trans_{\Vel_m} \psi}{L^2(\O\times(0,T))}
		+\norm{\transh_{\Vel_m} \psi}{L^2(\O\times(0,T))}
		\le C\norm{\psi}{L^2(\O\times(0,T))}.
		\ee
		\item The dual operators $\trans_{\Vel_m}^*$ and $\transh_{\Vel_m}^*$ of $\trans_{\Vel_m}$ and $\trans_{\Vel_m}$
		are given by: for $\psi\in L^2(\O\times(0,T))$,
		\be\label{def.taumstar}
		\begin{aligned}
			&\trans_{\Vel_m}^*\psi=\phi\trans_{-\Vel_m}\left(\frac{\psi}{\phi}\right) + R_m\trans_{-\Vel_m}\psi\\
			&\transh_{\Vel_m}^*\psi=\phi\transh_{-\Vel_m}\left(\frac{\psi}{\phi}\right) + \widehat{R}_m\transh_{-\Vel_m}\psi
		\end{aligned}
		\ee
		where $R_m,\widehat{R}_m\in L^\infty(\O\times(0,T))$ and, over each interval $[t^{(n)},t^{(n+1)}]$, $R_m,\widehat{R}_m$ are bounded by $\dtDisc\phi_*^{-1}\Gamma_\div C_1(T)$.
		\item If $f_m\to f$ strongly (resp. weakly) in $L^2(\O\times(0,T))$ as $m\to\infty$, then
		$\trans_{\Vel_m}f_m\to f$ and $\transh_{\Vel_m}f_m\to f$ strongly (resp. weakly) in $L^2(\O\times(0,T))$.
	\end{enumerate}
\end{lemma}

\begin{proof}
	
	We only prove the results for  $\trans_{\Vel_m}$, as the proof for $\transh_{\Vel_m}$ follows by simply replacing $F^{(n+1)}_{t^{(n+1)}-t^{(n)}}(\y)$ by $F^{(n+1)}_{t^{(n+1)}-t}(\y)$. In the first two steps, we drop the index $m$ in $\Vel_m$ and $N_m$ for simplicity of notation.
	
	\medskip
	
	\textbf{Step 1}: bound on the norms of $\trans_{\Vel}$ and $\transh_{\Vel}$.
	
	By a change of variables and invoking \eqref{JF.bd}, there is $C$ not depending on $m$, $s\in[-T,T]$
	or $n\in \{0,\ldots,N-1\}$ such that, for all $h\in L^2(\O)$,
	$\Vert h(F^{(n+1)}_s(\cdot))\Vert_{L^2(\O)}\le C\Vert h\Vert_{L^2(\O)}$.
	Estimate \eqref{est.L2.trans} easily follows from this.
	
	\medskip
	
	\textbf{Step 2}: description of the dual operator.
	
	A change of variables yields, for any $\phy,\psi\in L^2(\O\times(0,T))$,
	\begin{align}
	\int_{\O\times(0,T)}{}&(\trans_{\Vel} \phy)(\x,t)\psi(\x,t)d\x dt\nonumber\\
	={}&
	\sum_{n=0}^{N-1}\int_{t^{(n)}}^{t^{(n+1)}}\int_\O \phy(F^{(n+1)}_{t^{(n+1)}-t^{(n)}}(\x),t)\psi(\x,t)d\x dt\nonumber\\
	={}&
	\sum_{n=0}^{N-1}\int_{t^{(n)}}^{t^{(n+1)}}\int_\O \phy(\y,t)\psi(F^{(n+1)}_{t^{(n)}-t^{(n+1)}}(\y),t)
	|JF^{(n+1)}_{t^{(n)}-t^{(n+1)}}(\y)|d\y dt.
\label{def.taums}
	\end{align}
	Relation \eqref{JacobianInt} and Estimate \eqref{JF.bd} shows that
	\be\label{flow.Rm}
	|JF^{(n+1)}_{t^{(n)}-t^{(n+1)}}(\y)|=\frac{\phi(\y)}{\phi(F^{(n+1)}_{t^{(n)}-t^{(n+1)}}(\y))}+R(\y,t^{(n)})
	\ee
	with $|R(\y,t^{(n)})|\le \dtDisc \phi_*^{-1}\Gamma_\div C_1(T)$.
	Since $t\mapsto F^{(n+1)}_{t-t^{(n+1)}}(\y)$ is the flow corresponding to $-\Vel$, Relations
	\eqref{def.taums} and \eqref{flow.Rm} then yield \eqref{def.taumstar} for $\trans_{\Vel}^*$.
	
	\medskip
	
	\textbf{Step 3}: proof of the strong convergence.
	
	For simplicity of notation, denote $\norm{\cdot}{2}= \norm{\cdot}{ L^2(\O\times(0,T))}$. 
	Assume that $f_m\to f$ strongly in $L^2(\O\times(0,T))$, and
	let $f^\eps$ be a smooth approximation of $f$ such that $\norm{f-f^\eps}{2}\leq \eps$. 
	The triangle inequality and \eqref{est.L2.trans} yield
	\begin{align*}
	\norm{\trans_{\Vel_m}f_m-f}{2}
	\leq{}&\norm{\trans_{\Vel_m}(f_m-f)}{2}+\norm{\trans_{\Vel_m}(f-f^\eps)}{2}+\norm{\trans_{\Vel_m} f^\eps-f^\eps}{2}\\
	&+\norm{f^\eps-f}{2}\\
	\le{}&C\norm{f_m-f}{2}+(C+1)\eps + \norm{\trans_{\Vel_m} f^\eps-f^\eps}{2}.
	\end{align*}
	Invoking Lemma \ref{lem:flow.cont} with $\alpha = 2$, $r=\infty$ and $f^\eps(\cdot,t)$ instead of $f$
	gives gives $C'$ not depending on $m$ or $\eps$ such that, if $F^{(n+1)}_{m,t}$ is the flow for
	the velocity $\Vel_m^{(n+1)}$,
	\begin{align*}
	\norm{\trans_{\Vel_m} f^\eps-f^\eps}{2}^2={}&\sum_{n=0}^{N_m-1}\int_{t^{(n)}}^{t^{(n+1)}}\norm{f^\eps(F^{(n+1)}_{m,t^{(n+1)}-t^{(n)}}(\cdot),t)-f^\eps(\cdot,t)}{L^2(\O)}^2dt\\
	\le{}&C'\dt_m^2\sum_{n=0}^{N_m-1}\int_{t^{(n)}}^{t^{(n+1)}}\norm{\Vel_m^{(n+1)}}{L^2(\O)}^2\norm{\nabla f^\eps(\cdot,t)}{L^\infty(\O)}^2 dt\\
	={}&C'\dt_m^2\norm{\Vel_m}{2}^2\norm{\nabla f^\eps}{L^\infty(\O\times(0,T))}^2.
	\end{align*}
	Hence,
	\[
	\norm{\trans_{\Vel_m}f_m-f}{2}\leq C \norm{f_m-f}{2}+ (1+C)\eps+\sqrt{C'}\dt_m\norm{\Vel_m}{2}\norm{\nabla f^\eps}{L^\infty(\O\times(0,T))}.
	\]
	Taking the superior limit as $m\to\infty$ and using the  boundedness of $(\Vel_m)_{m\in\N}$
	in $L^2(\O\times(0,T))$ thus yields
	$\limsup_{m\to\infty}\norm{\trans_{\Vel_m} f_m-f}{2}\le (1+C)\eps$. Letting $\eps\to 0$ concludes
	the proof that $\trans_{\Vel_m}f_m\to f$ strongly in $L^2(\O\times(0,T))$.
	
	\medskip
	
	\textbf{Step 4}: proof of the weak convergence.
	
	Assume that $f_m\to f$ weakly in $L^2(\O\times(0,T))$. Then, for all $\psi\in L^2(\O\times(0,T))$,
	\begin{align}
	\int_{\O\times(0,T)} (\trans_{\Vel_m} f_m-f)\psi ={}& \int_{\O\times(0,T)} \trans_{\Vel_m}(f_m-f)\psi +
	\int_{\O\times(0,T)} (\trans_{\Vel_m} f-f)\psi \nonumber\\
	={}& \int_{\O\times(0,T)} (f_m-f)\trans_{\Vel_m}^*\psi +\int_{\O\times(0,T)} (\trans_{\Vel_m} f-f)\psi.
	\label{weak.conv.Tm}
	\end{align}
	Since $\psi/\phi\in L^2(\O\times(0,T))$, the formula \eqref{def.taumstar}, the fact that
	$R_m\to 0$ in $L^\infty(\O\times(0,T))$, the estimate \eqref{est.L2.trans} and the
	result of Step 3 applied to $-\Vel_m$ instead of $\Vel_m$
	show that $\trans_{\Vel_m}^*\psi\to \psi$ strongly in $L^2(\O\times(0,T))$ as $m\to\infty$.
	Hence, the first term in the right-hand side of \eqref{weak.conv.Tm}
	tends to $0$ since $f_m-f\to 0$ weakly $L^2(\O\times(0,T))$. The second term in the right-hand side of \eqref{weak.conv.Tm}
	also converges to $0$ since, by Step 3 (applied to $f_m=f$ for all $m$),
	$\trans_{\Vel_m} f-f\to 0$ in $L^2(\O\times(0,T))$. The proof that $\trans_{\Vel_m}f_m\to f$
	weakly in $L^2(\O\times(0,T))$ is therefore complete.
\end{proof}

	\section{A Priori Estimates}\label{sec:a-priori-est}
		Throughout this section, $A\lesssim B$ means that $A\leq CB$, where $C$ is a constant depending only on the quantities $|\O|$, $T$, $\phi_*$, $\phi^*$, $\alpha_A$, $\alpha_\diffTens$, $\Lambda_A$, $\Lambda_\diffTens$, $M_{q^-}$, $M_{q^+}$, $M_t$, $M_F$, $M_\div$, $\sup_{m\in\N}C_{\discP_m}$, $\sup_{m\in\N}C_{\discC_m}$
appearing in Assumptions \eqref{assump.global} and \ref{hyp:GDs}--\ref{hyp:conv.u}
($C_{\discP_m}$ and $C_{\discC_m}$ are given by \eqref{def:CD}).
Likewise, in the proofs, $C$ denotes a generic constant that can change from one line to the other, but only depends on the aforementioned parameters.

We also consider that $(p_m,c_m)$ is a solution to the GDM--ELLAM scheme with $(\discP,\discC^T)=(\discP_m,\discC^T_m)$
and we drop the index $m$ for legibility. Let 
$\U_\discP=-\frac{\K}{\mu(\widetilde{\Pi}_\discC c)}\nabla_{\discP} p$.

\begin{lemma}[Estimates on the pressure]\label{estimates_press}
The following estimate holds:
\[
\norm{\Pi_{\discP} p}{L^{\infty}(0,T;L^2(\O))}+
\norm{\grad_{\discP} p}{L^{\infty}(0,T;L^2(\O))}+
\norm{\U_\discP}{L^{\infty}(0,T;L^2(\O))}\lesssim 1.
\]
\end{lemma}

\begin{proof}
	Setting $z=p^{(n+1)}$ in the gradient scheme \eqref{GSpress}, we get:
	\be \nonumber
		\int_{\O} A(\x,\Pi_{\discC} c^{(n)}) \grad_{\discP} p^{(n+1)} \cdot \grad_{\discP} p^{(n+1)} = \int_{\O} (q^+_n-q^-_n) \Pi_{\discP} p^{(n+1)}.
	\ee
	Using \eqref{hyp:viscosity} for the left hand side, followed by Cauchy--Schwarz' inequality 
\be\label{est.pm}
	\norm{\grad_{\discP} p^{(n+1)}}{L^2(\O)}^2 \lesssim \norm{q^+_n-q^-_n}{L^2(\O)} \norm{\Pi_{\discP} p^{(n+1)} }{L^2(\O)}
\lesssim\norm{\grad_{\discP} p^{(n+1)}}{L^2(\O)}
\ee
where we used 
\be\label{poinc}
\norm{\Pi_{\discP} p^{(n+1)} }{L^2(\O)}\lesssim \norm{p^{(n+1)}}{\discP}=\norm{\grad_{\discP} p^{(n+1)}}{L^2(\O)}
\ee
since $\int_\O \Pi_{\discP} p^{(n+1)} =0$. Equation \eqref{est.pm} proves the
estimate on $\grad_{\discP} p$ which gives the bound on $\U_\discP$ (owing to \eqref{hyp:viscosity}) and, using \eqref{poinc} once more, provides the estimate
on $\Pi_{\discP} p$. \end{proof}

	\begin{lemma}[Estimates on the concentration] \label{space_estimates_conc}
		The following estimate holds:
	\[
	\norm{\PiDc c}{L^\infty(0,T;L^2(\O))}+\norm{(1+|\U_\discP|)^{1/2} \nabla_{\discC} c}{L^2(0,T;L^2(\Omega))}\lesssim 1
+\norm{\PiDc \ICinterp_\discC c_{\rm ini}}{L^2(\O)}.
\]
As a consequence, $\norm{\gradDc c}{L^2(0,T;L^2(\O))}\lesssim 1+\norm{\PiDc \ICinterp_\discC c_{\rm ini}}{L^2(\O)}$.
	\end{lemma}

\begin{proof}
	 Denote $Y_n= \norm{\PiDc c^{(n)} \sqrt{\phi}}{L^2(\O)}$. 
	The gradient scheme \eqref{GSconc} with $z=c^{(n+1)}$ yields
	\begin{equation}\nonumber
	\begin{aligned}
	Y_{n+1}^2&-\int_{\O} \phi \PiDc c^{(n)} v(t^{(n)}) +\dtDisc \int_{\O} D(\x,\U_\discP^{(n+1)}) \gradDc c^{(n+1)} \cdot \gradDc c^{(n+1)} \\
	&+ w\dtDisc \int_{\O} \PiDc c^{(n)} v(t^{(n)}) q_n^- + (1-w)\dtDisc \int_{\O} (\PiDc c^{(n+1)})^2 q_{n+1}^- \\
	&= w\dtDisc \int_{\O} q_n^+ v(t^{(n)}) +(1-w) \dtDisc \int_{\O} q_{n+1}^+\PiDc  c^{(n+1)}=:\Delta.
	\end{aligned}
	\end{equation}
	Now, by Cauchy-Schwarz, recalling that $0\le w\le 1$ and that $|q_n^-/\sqrt{\phi}|\le M_{q^-}/\sqrt{\phi_*}$,
and using the coercivity property of the diffusion tensor $\diffTens$,
	\begin{equation}\nonumber
	\begin{aligned}
	\Delta 
	&\geq Y_{n+1}^2 - Y_{n}\norm{v(t^{(n)})\sqrt{\phi}}{L^2(\O)}+\alpha_\diffTens \dtDisc  \norm{(1+|\U_\discP^{(n+1)}|)|\gradDc c^{(n+1)}|^2}{L^1(\O)} \\
	&\quad 
	-\dfrac{M_{q^-}}{\sqrt{\phi_{*}}}\dtDisc Y_{n}\norm{v(t^{(n)})}{L^2(\O)}.
	\end{aligned}
	\end{equation}
	Consider the term $Y_{n}\norm{v(t^{(n)})\sqrt{\phi}}{L^2(\O)}$  in the right hand side of the inequality. 
Estimate \eqref{Pv2Ineq} with $w(\x,t)=v(\x,t)^2$ and $s=\dtDisc$ (so that $v(t^{(n+1)}-s)=v(t^{(n)})$), followed by Young's inequality, give, for any $\eps>0$,
	\begin{equation}\label{znvn}
	\begin{aligned}
	Y_{n}\norm{v(t^{(n)})\sqrt{\phi}}{L^2(\O)} &\leq Y_{n}Y_{n+1}\sqrt{1+C\dtDisc} 
	\leq Y_{n}Y_{n+1}(1+C\dtDisc)\\
	&\leq  
	\dfrac{1}{2}Y_{n}^2+\dfrac{1}{2}Y_{n+1}^2+\dfrac{C^2\dtDisc}{2\eps}Y_{n}^2+\dfrac{\dtDisc\eps}{2} Y_{n+1}^2.
	\end{aligned}
	\end{equation}
	Using \eqref{v2Ineq},
	\begin{equation}\label{ynvn}
	Y_{n}\norm{v(t^{(n)})}{L^2(\O)} \leq CY_nY_{n+1} 
	\leq \dfrac{C^2}{2\eps}Y_{n}^2+\dfrac{\eps}{2}Y_{n+1}^2.
	\end{equation}
	Using \eqref{znvn} together with  \eqref{ynvn}, we then have
	\begin{equation}\nonumber
	\begin{aligned}
	\Delta  \geq{}& Y_{n+1}^2 - \left(	\dfrac{1}{2}Y_{n}^2+\dfrac{1}{2}Y_{n+1}^2+\dfrac{C^2\dtDisc}{2\eps}Y_{n}^2+\dfrac{\dtDisc\eps}{2} Y_{n+1}^2\right)\\
	& +\alpha_\diffTens \dtDisc  \norm{(1+|\U_\discP^{(n+1)}|)|\gradDc c^{(n+1)}|^2}{L^1(\O)}\\
	& -\dfrac{M_{q^-}}{\sqrt{\phi_{*}}}\dtDisc \left(\dfrac{C^2}{2\eps}Y_{n}^2+\dfrac{\eps}{2}Y_{n+1}^2\right),
	\end{aligned}
	\end{equation}
	which implies that
	\begin{multline}
	\dfrac{1}{2}Y_{n+1}^2-\dfrac{1}{2}Y_{n}^2 +\alpha_\diffTens \dtDisc  \norm{(1+|\U_\discP^{(n+1)}|)|\gradDc c^{(n+1)}|^2}{L^1(\O)}   \\
	\lesssim \Delta +\dfrac{\dtDisc}{\eps}Y_{n}^2+\eps \dtDisc Y_{n+1}^2.
	\label{est.c.1}
	\end{multline}
	Now, using the boundedness of $q^+$, Young's inequality, the fact that $w\in [0,1]$ and
	\eqref{v2Ineq} with $w(\x,t)=v(\x,t)^2$ and $s=\dtDisc$ ,
	\begin{equation}\nonumber
	\begin{aligned}
	\Delta 
	&\lesssim \dtDisc  \left(\norm{v(t^{(n)})}{L^{2}(\O)} + \norm{ \PiDc c^{(n+1)}}{L^{2}(\O)}\right) \\
	&\lesssim \dtDisc  \left[\dfrac{1}{\eps}+\eps\norm{v(t^{(n)})}{L^{2}(\O)}^2 +\eps Y_{n+1}^2\right] 
	\lesssim  \dfrac{\dtDisc}{\eps} +\dtDisc \eps Y_{n+1}^2. \\
	\end{aligned}
	\end{equation}
	Combining with \eqref{est.c.1}, we find 
	\begin{multline*}
	\dfrac{1}{2}Y_{n+1}^2-\dfrac{1}{2}Y_{n}^2 +\alpha_\diffTens \dtDisc  \norm{(1+|\U_{\discP}^{(n+1)}|)|\gradDc c^{(n+1)}|^2}{L^1(\O)}\\ \lesssim   \dfrac{\dtDisc}{\eps} +\dfrac{\dtDisc}{\eps}Y_{n}^2+\eps \dtDisc Y_{n+1}^2,
	\end{multline*}
	which, upon taking a telescoping sum, yields
	\begin{equation}\nonumber
	\begin{aligned}
	\dfrac{1}{2} Y_{n+1}^2{}&-\dfrac{1}{2}Y_0^2 +\alpha_\diffTens \sum_{k=0}^{n} \dt^{(k+\frac{1}{2})}  \norm{(1+|\U_{\discP}^{(n+1)}|)|\gradDc c^{(n+1)}|^2}{L^1(\O)}\\
	\lesssim{}& \dfrac{1}{\eps} \sum_{k=0}^n \dt^{(k+\frac{1}{2})} +\dfrac{1}{\eps} \sum_{k=0}^n \dt^{(k+\frac{1}{2})} Y_k^2 + \eps\sum_{k=1}^{n+1} \dt ^{(k-\frac{1}{2})} Y_k^2\\
	\lesssim{}& \dfrac{1}{\eps} T +\dfrac{1}{\eps} \dt^{(\frac{1}{2})} Y_0^2 + \eps \dt ^{(n+\frac{1}{2})} Y_{n+1}^2+\left(\dfrac{1}{\eps} + \eps\right)\sum_{k=1}^n (\dt^{(k+\frac{1}{2})}+\dt ^{(k-\frac{1}{2})}) Y_k^2.
	\end{aligned}
	\end{equation}
	Denoting by $C$ the hidden multiplicative constant in the last $\lesssim$ above,
	choose $\eps =1/(4CT)$ to absorb the term $\eps \dt ^{(n+\frac{1}{2})} Y_{n+1}^2$
	in the left-hand side. Since $\eps$ depends only on fixed quantities, we no longer make it explicit
	and it disappears into the $\lesssim$ symbols.
	Setting $\dt^{(-\frac{1}{2})}=0$ the term $\dt^{(\frac{1}{2})} Y_0^2$ can be integrated in the last sum and we find
	\be\label{est.before.gronwall}
	Y_{n+1}^2 + \norm{(1+|\U_\discP|)|\gradDc c|^2}{L^1(\O\times(0,t^{(n+1)})}	\lesssim  1+Y_0^2+\sum_{k=0}^n (\dt^{(k+\frac{1}{2})}+\dt ^{(k-\frac{1}{2})}) Y_k^2.
	\ee

Dropping for a moment the second term in the left-hand side, and letting $C$ denote the hidden multiplicative constant in $\lesssim$, a discrete Gronwall's inequality \cite[Section 5]{H09-DiscGronwall} yields, for any $n=0,\ldots,N-1$,
	\[
	Y_{n+1}^2 \le C(1+Y_{0}^2)\exp\Big(\sum_{k=0}^{n} C(\dt^{(k+\frac{1}{2})}+\dt ^{(k-\frac{1}{2})})\Big)
	\le C(1+Y_{0}^2)\exp(2CT).
	\]
By noticing that $Y_0\le \sqrt{\phi^*} \norm{\PiDc c^{(0)}}{L^2(\O)}=\sqrt{\phi^*} \norm{\PiDc \ICinterp_\discC c_{\rm ini}}{L^2(\O)}$,
this proves the estimate on $\norm{\PiDc c}{L^\infty(0,T;L^2(\O))}$. Plugging this estimate in
	\eqref{est.before.gronwall} with $n=N-1$ yields the estimate on $\norm{(1+|\U_\discP|)^{1/2} \nabla_{\discC} c}{L^2(0,T;L^2(\Omega))}$
	which, in turn, trivially provides a bound on $\norm{\nabla_{\discC} c}{L^2(0,T;L^2(\Omega))}$.
\end{proof}

\begin{remark}[Estimate of the advection--reaction terms]
A formal integration-by-parts shows that, if $\darcyU$ satisfies \eqref{pressure},
\[
\int_\O \div(c\darcyU)c+\int_\O q^-c^2=\frac{1}{2}\int_\O (q^++q^-)c^2\ge 0.
\]
When using $c$ as a test function in the continuous equation, the advection and reaction terms thus
combine to create a non-negative quantity that can simply be discarded from the estimates (which
thus hold under very weak assumptions on $q^\pm$). This can be reproduced at the discrete
level for upwind discretisations \cite{ckm15,CD-07}. However, the structure of the ELLAM
discretisation does not seem to lend itself to such an easy estimate of the advection--reaction
terms, which is why the proof of Lemma \ref{space_estimates_conc} is a bit technical, and requires
the boundedness of $q^\pm$ (to bound the Jacobian of the changes of variables -- note that we do
not require a bound on $\darcyU$ itself, though).

\end{remark}

A crucial step in the convergence proof is to establish the strong compactness
of $\PiDc c$. This is done by using a discrete version of the Aubin--Simon theorem.
The gradient estimates in Lemma \ref{space_estimates_conc} provides the compactness
in space, which must be complemented by some sort of boundedness (in a dual norm) of the
discrete time-derivative of $c$. Establishing this boundedness is the purpose of the following lemma. 
A dual norm $\norm{\cdot}{\star,\phi,\discC}$ is defined on $\PiDc(X_{\discC})$ the following way:
\[
\begin{aligned}
&\forall w \in \PiDc(X_{\discC}) & \\
& \norm{w}{\star,\phi,\discC} := \sup \left\{ \int_{\O}\phi w \PiDc v\,:\, v\in X_{\discC}\,,\; \norm{\gradDc v}{L^{4}(\O)} +\norm{\PiDc v}{L^\infty(\O)}=1 \right\}.
\end{aligned}
\]
It can easily be checked that this is indeed a norm (if $w\not=0$,
write $w=\PiDc z$, take $v=z/N$ where $N=\norm{\gradDc z}{L^{4}(\O)} +\norm{\PiDc z}{L^\infty(\O)}>0$,
and notice that $\norm{w}{\star,\phi,\discC}\ge \int_\O \phi w(\x)\PiDc v(\x)d\x=N^{-1}\norm{\sqrt{\phi}w}{L^2(\O)}^2$).

\begin{lemma}\label{time_derivative_estimate}
Defining the discrete time derivative of $c$ by 
\[
\delta_{\discC} c(t)=\frac{\PiDc c^{(n+1)}-\PiDc c^{(n)}}{\dtDisc}\mbox{ for all $t\in (t^{(n)},t^{(n+1)})$
and all $n=0,\ldots,N-1$},
\]
we have
\[
\int_0^T \norm{\delta_{\discC} c}{\star,\phi,\discC}^2 dt \lesssim 1+\norm{\PiDc \ICinterp_\discC c_{\rm ini}}{L^\infty(\O)}^2.
\]
\end{lemma}
\begin{proof}
Take $z\in X_{\discC}$ arbitrary in \eqref{GSconc}. Subtract and add $\int_\O \phi \PiDc c^{(n)} \PiDc z$ to get
	\begin{equation}\nonumber
\begin{aligned}
\int_{\O}{}& \phi (\PiDc c^{(n+1)} -\PiDc c^{(n)} )\PiDc z \\
={}& -\int_{\O} \phi \PiDc c^{(n)} (\PiDc z-v(t^{(n)})) 
-\dtDisc \int_{\O} D(\x,\U_\discP^{(n+1)}) \gradDc c^{(n+1)} \cdot \gradDc z \\
&- w\dtDisc \int_{\O} \PiDc c^{(n)} v(t^{(n)})q_n^-
- (1-w)\dtDisc \int_{\O} \PiDc c^{(n+1)} \PiDc zq_{n+1}^- \\
&+ w\dtDisc \int_{\O} q_n^+v(t^{(n)})
+(1-w) \dtDisc \int_{\O} q_{n+1}^+\PiDc z.
\end{aligned}
\end{equation}
The terms on the right hand side of the equation are referred to as $T_1,T_2,\ldots,T_6$, respectively.
For the term $T_1$,  recall that $v(\x,t^{(n)})=\PiDc z(F_{\dt^{(n+1/2)}}(\x))$.
If $n=0$, recalling that $c^{(0)}=\ICinterp_\discC c_{\rm ini}$ and applying \ref{hyp:PiD.flow} shows that
\be \label{est:T1:n0}
\begin{aligned}
	|T_1|&\lesssim \norm{\Pi_{\discC} \ICinterp_\discC c_{\rm ini}}{L^\infty(\O)}\norm{\PiDc z-\PiDc z(F_{\dt^{(1/2)}})}{L^1(\O)}\\
		&\lesssim \delta t^{(\frac{1}{2})} \norm{\Pi_{\discC}\ICinterp_\discC c_{\rm ini}}{L^\infty(\O)}\norm{\discDarcyU[(1)]}{L^2(\O)} \norm{\grad_{\discC} z}{L^2(\O)}.
\end{aligned}
\ee
If $n\neq 0$, a
change of variables yields
\[
\begin{aligned}
-T_1
={}& \int_{\O} \phi \PiDc c^{(n)}  \PiDc z\\
&-\int_{\O} \phi\!\left(F_{-\dt^{(n+1/2)}}(\x)\right) \PiDc c^{(n)}\!\!\left(F_{-\dt^{(n+1/2)}}(\x)\right) \PiDc z(\x)\left|JF_{-\dt^{(n+1/2)}}(\x)\right| d\x.
\end{aligned}
\]
Applying \eqref{JacobianInt} with $s=-\dtDisc$, we can thus write $-T_1=T_{11}-T_{12}$
with
\begin{equation}\nonumber
\begin{aligned}
T_{11}&= \int_{\O} \phi \PiDc c^{(n)}  \PiDc z
-\int_{\O} \phi(\x) \PiDc c^{(n)}\left(F_{-\dt^{(n+1/2)}}(\x)\right)  \PiDc z (\x)d\x\\
T_{12}&=\int_{\O}\bigg[\PiDc c^{(n)}\left(F_{-\dt^{(n+1/2)}}(\x)\right)
\PiDc z(\x) \\
&\qquad\qquad\times\int_0^{-\dtDisc} |JF_t(\x)| (\divg \discDarcyU) \circ F_t(\x) dt\bigg] d\x.
\end{aligned}
\end{equation}
Using \ref{hyp:PiD.flow} leads to
\[
\begin{aligned}
|T_{11}| &\leq \int_\O \left|\phi \PiDc z \left( \PiDc c^{(n)} - \PiDc c^{(n)}(F_{-\dt^{(n+1/2)}})  \right) \right|\\
&\lesssim \dtDisc \norm{\PiDc z}{L^\infty(\O)}\norm{\discDarcyU}{L^2(\O)}\norm{\nabla_\disc c^{(n)}}{L^2(\O)}.
\end{aligned}
\]

The boundedness of $\div  \discDarcyU$ in \ref{hyp:hdiv} and of $|JF_t|$ (see \eqref{JF.bd}) yield,
by a change of variables,
\[
\begin{aligned}
|T_{12}|
\lesssim{}& \dtDisc  \norm{\PiDc c^{(n)}(F_{-\dt^{(n+1/2)}})}{L^2(\O)} \norm{\PiDc z}{L^2(\O)}\\
\lesssim{}& \dtDisc  \norm{\PiDc c^{(n)}}{L^2(\O)} \norm{\PiDc z}{L^2(\O)}.
\end{aligned}
\]

For the term $T_{2}$, the property \eqref{hyp:diff} of the diffusion tensor $\diffTens$ and H\"older's inequality with exponents 4, 2 and 4 give
\begin{equation}\nonumber
\begin{aligned}
|T_2| &\lesssim \dtDisc \int_{\O} \sqrt{1+|\U_\discP^{(n+1)}|} \left(\sqrt{1+|\U_\discP^{(n+1)}|} \,|\gradDc c^{(n+1)}|\right) |\gradDc z|\\
&\lesssim \dtDisc\norm{1+|\U_\discP^{(n+1)}|}{L^2(\O)}^{\frac{1}{2}} \norm{(1+|\U_\discP^{(n+1)}|)^{\frac{1}{2}} |\gradDc c^{(n+1)}| }{L^2(\O)} \norm{\gradDc z}{L^4(\O)}.
\end{aligned}
\end{equation}
The terms $T_3$ to $T_6$ are estimated by using the Cauchy--Schwarz inequality:
\begin{align*}
 |T_3| \lesssim{}& \dtDisc  \norm{\PiDc c^{(n)}}{L^2(\O)} \norm{v(t^{(n)})}{L^2(\O)},\\
|T_{4}| \lesssim{}&  \dtDisc \norm{\PiDc c^{(n+1)}}{L^2(\O)}\norm{\PiDc z}{L^2(\O)},\\
|T_{5}+T_6| \lesssim{}&  \dtDisc \norm{v(t^{(n)})}{L^2(\O)}+ \dtDisc\norm{\PiDc z}{L^2(\O)}
\lesssim \dtDisc\norm{\PiDc z}{L^2(\O)}
\end{align*}
(we used \eqref{v2Ineq} with $w=v^2$ and $s=\dtDisc$ to obtain $\norm{v(t^{(n)})}{L^2(\O)} \lesssim \norm{\Pi_{\discC} z}{L^2(\O)}$). For $n\neq 0$, combining the estimates from $T_{1}$ to $T_{6}$ leads to
\begin{align}
\int_{\O}\phi {}&(\PiDc c^{(n+1)}-\PiDc c^{(n)}) \PiDc z \nonumber\\
\lesssim{}& \dtDisc  \norm{ \PiDc z}{L^\infty(\O)} \norm{\discDarcyU}{L^2(\O)}\norm{\gradDc c^{(n)} }{L^2(\O)}
\label{term.to.replace}\\
&+ \dtDisc\norm{\PiDc c^{(n)}}{L^2(\O)}\norm{\PiDc z}{L^2(\O)}\nonumber\\
&+\dtDisc \norm{1+|\U_\discP^{(n+1)}|}{L^2(\O)}^{\frac{1}{2}} \norm{(1+|\U_\discP^{(n+1)}|)^{\frac{1}{2}} |\gradDc c^{(n+1)}| }{L^2(\O)}
\norm{\gradDc z}{L^4(\O)}\nonumber\\
&
+\dtDisc\norm{\PiDc c^{(n+1)}}{L^2(\O)}\norm{\PiDc z}{L^2(\O)}+\dtDisc\norm{\PiDc z}{L^2(\O)}.
\nonumber
\end{align}
Divide both sides by $\dtDisc$
and take the supremum over all $z\in X_{\discC}$ with $\norm{\gradDc z}{L^{4}(\O)} +\norm{\PiDc z}{L^\infty(\O)}=1$ to obtain, for all $n=1,\ldots,N-1$ and $t\in (t^{(n)},t^{(n+1)})$,
\begin{multline}\label{final.dtc}
\norm{\delta_{\discC}c(t)}{\star,\phi,\discC} \lesssim  \norm{\discDarcyU}{L^2(\O)}\norm{\gradDc c^{(n)} }{L^2(\O)} 
+ \norm{\PiDc c^{(n)}}{L^2(\O)}
+\norm{\PiDc c^{(n+1)}}{L^2(\O)}
\\
+ \norm{1+|\U_\discP^{(n+1)}|}{L^2(\O)}^{\frac{1}{2}} \norm{(1+|\U_\discP^{(n+1)}|)^{\frac{1}{2}} \gradDc c^{(n+1)} }{L^2(\O)}
+1.
\end{multline}
Square this, integrate for $t\in (t^{(n)},t^{(n+1)})$ and sum over $n=1,\ldots,N-1$. The assumption on the time steps in \ref{hyp:GDs}
ensures that
\begin{multline*}
\sum_{n=1}^{N-1}\dt^{(n+1/2)}\norm{\nabla_\disc c^{(n)}}{L^2(\O)}^2
\lesssim \sum_{n=1}^{N-1}\dt^{(n-1/2)}\norm{\nabla_\disc c^{(n)}}{L^2(\O)}^2\\
=\sum_{n=0}^{N-2}\dt^{(n+1/2)}\norm{\nabla_\disc c^{(n+1)}}{L^2(\O)}^2
\le \norm{\nabla_\discC c}{L^2(\O\times(0,T))}^2
\end{multline*}
(and similarly for the terms involving $\PiDc c^{(n)}$),
so that
\begin{multline}\label{est.dtc.1}
\int_{t^{(1)}}^T \norm{\delta_{\discC}c(t)}{\star,\phi,\discC}^2dt 
\lesssim \norm{\discDarcyU[]}{L^\infty(0,T;L^2(\O))}^2\norm{\gradDc c }{L^2(\O\times(0,T))}^2 
+\norm{\PiDc c}{L^2(\O\times(0,T))}^2\\
+ \norm{1+|\U_\discP|}{L^\infty(0,T;L^2(\O))}
\norm{(1+|\U_\discP|)^{\frac{1}{2}} \gradDc c }{L^2(\O\times (0,T))}^2
+1.
\end{multline}
To estimate $\int_0^{t^{(1)}} \norm{\delta_{\discC}c(t)}{\star,\phi,\discC}^2dt$,
we come back to \eqref{term.to.replace} with $n=0$. The first term
in the right-hand side of this inequality must be replaced by the right-hand side of
\eqref{est:T1:n0}, and thus the first term in \eqref{final.dtc} is replaced
by $\norm{\Pi_{\discC}\ICinterp_\discC c_{\rm ini}}{L^\infty(\O)}\Vert \discDarcyU[(1)]\Vert_{L^2(\O)}$.
Hence,
\begin{multline}\label{est.dtc.2}
\int_0^{t^{(1)}} \norm{\delta_{\discC}c(t)}{\star,\phi,\discC}^2dt 
\lesssim \dt^{(1/2)}\norm{\Pi_{\discC}\ICinterp_\discC c_{\rm ini}}{L^\infty(\O)}^2\norm{\discDarcyU[(1)]}{L^2(\O)}^2\\
+\dt^{(1/2)}\norm{\PiDc \ICinterp_{\discC} c_{\rm ini}}{L^2(\O)}^2
+\norm{\PiDc c}{L^2(\O\times(0,T))}^2\\
+ \norm{1+|\U_\discP|}{L^\infty(0,T;L^2(\O))}
\norm{(1+|\U_\discP|)^{\frac{1}{2}} \gradDc c }{L^2(\O\times (0,T))}^2
+1.
\end{multline}

The reason for separating the case $n\not=0$ from the case $n=0$ is that,
for $n=0$, \eqref{term.to.replace} involves $\nabla_\discC c^{(0)}=\nabla_\discC \ICinterp_\discC c_{\rm ini}$
on which no bound has been imposed.
The proof is completed by adding together \eqref{est.dtc.1} and \eqref{est.dtc.2}, and by invoking
Assumption \ref{hyp:conv.u} and Lemmas \ref{estimates_press} and \ref{space_estimates_conc}.
\end{proof}

\section{Proof of the main theorem} \label{sec:comp-conv}

At each time step, \eqref{GSpress} and \eqref{GSconc} are square linear
equations on $p^{(n+1)}$ and $c^{(n+1)}$. The estimates of Lemma \ref{estimates_press} and \ref{space_estimates_conc},
together with the definition of the norms in $X_{\discP}$ and $X_{\discC}$,
show that any solutions to these linear systems remains bounded. Hence,
the matrices associated with these systems do not have any kernel, which ensures
the existence and uniqueness of $(p,c)$ solution to the GDM--ELLAM scheme.

We now establish the compactness of $(\Pi_{\discC_m}c_m)_{m\in\N}$, which is
essential to proving the convergence of the pressure. Once this latter is establish,
we conclude the proof by dealing with the convergence of the concentration.

\subsection{Compactness and initial convergence of $\Pi_{\disc_m}c_m$}\label{sec:compactness}

\begin{theorem} \label{relCompact}
Under the assumptions and notations of Theorem \ref{th:main.convergence},
the sequence $(\Pi_{\discC_m}c_m)_{m\in\N}$ is relatively compact in $L^2(0,T;L^2(\O))$. 
\end{theorem}
	\begin{proof}
	The idea is to apply Theorem \ref{disc-Aubin-Simon} with $X_m = \Pi_{\discC_m}(X_{\discC_m})$ equipped with the norm $\norm{u}{X_m}=\min \{ \norm{w}{\discC_m}: w \in X_{\discC_m}\mbox{ s.t.\ } \Pi_{\discC_m} w = u \}$ and $Y_m=X_m$ with the norm $\norm{\cdot}{Y_m}=\norm{\cdot}{\star,\phi,\discC_m}$. 
	 
	 \medskip
	 
Let us show that $(X_m,Y_m)_{m\in\N}$ is compactly--continuously embedded in $L^2(\O)$ (Definition \ref{def.cc}).
Item \ref{cc.it1} follows by the compactness of $(\discC_m)_{m\in\N}$, see Definition \ref{def:propGDs}.
Take now $(u_m)_{m\in\N}$ as prescribed in Item \ref{cc.it2} and let $u$ be the limit in $L^2(\O)$ of this sequence.
Let  $\phy\in C_c^\infty(\O)$
and consider the interpolant $\sinterp_{\discC_m}$ given by Assumption \ref{hyp:smooth.interp}.
Then $\norm{\Pi_{\discC_m}\sinterp_{\discC_m}\phy}{L^\infty(\O)}+
\norm{\nabla_{\discC_m}\sinterp_{\discC_m}\phy}{L^4(\O)}\le C_\phy$ for some $C_\phy>0$ not depending
on $m$, and thus, by definition of $\norm{\cdot}{Y_m}=\norm{\cdot}{\star,\phi,\discC_m}$,
	\begin{equation} \nonumber
	\left|\int_{\O} \phi u_m \dfrac{\Pi_{\discC_m} \sinterp_{\discC_m}\phy}{C_\phy} \right| \leq \norm{u_m}{Y_m}.
	\end{equation}
	Taking the limit as $m\rightarrow\infty$, we get $\int_{\O} \phi u \phy=0$.
	Since this is true for all $\phy\in C_c^\infty(\O)$, we deduce that $u=0$ as required. 
	
	\medskip
	
	We are left to show that the sequence $f_m=(\Pi_{\discC_m}c_m)_{m\in\N}$ satisfies the properties in Theorem \ref{disc-Aubin-Simon}. The first property is trivially satisfied by the definition $f_m$,
whereas the second and third one follow from Lemma \ref{space_estimates_conc} and
the definition of the norm $\norm{\cdot}{\discC_m}$ (Definition \ref{GDdef}).
The last property holds due to Lemma \ref{time_derivative_estimate}.
	
	Thus, we may use Theorem \ref{disc-Aubin-Simon} to conclude that the sequence $(\Pi_{\discC_m}c_m)_{m\in\N}$ is relatively compact in $L^2(0,T;L^2(\O))$. 
\end{proof}

Theorem \ref{relCompact} together with Lemma \ref{reg-lim-space-time} give $c \in L^2(0,T;H^1(\O))$ such that, up to a subsequence as $m\to\infty$, $\Pi_{\discC_m}c_m\rightarrow c$ strongly in $L^2((0,T)\times \O)$ and $\grad_{\discC_m} c_m \rightarrow \grad {c}$ weakly in $L^2((0,T)\times\O)^d$. From here on we always
consider subsequences that satisfy these convergences.
Let $\alpha_m:[0,T]\to\R$ be the piecewise affine map that
maps each interval $(t_m^{(n)},t_m^{(n+1)})$ onto $(t_m^{(n-1)},t_m^{(n)})$, for $n= 1,\ldots,N_m-1$.
That is, $\alpha_m(t)=t-(1-\frac{\dt_m^{(n-1/2)}}{\dt_m^{(n+1/2)}})(t-t^{(n)})-(t^{(n)}-t^{(n-1)})$
for $t\in (t^{(n)},t^{(n+1)})$.
Recalling the definition of $\widetilde{\Pi}_{\discC_m}c_m$
at the start of Section \ref{sec:a-priori-est}, we have $\widetilde{\Pi}_{\discC_m}c_m=\Pi_{\discC_m}c_m(\cdot,
\alpha_m(\cdot))$ on $\O\times(t^{(1)},T)$ and $\widetilde{\Pi}_{\discC_m}c_m=\Pi_{\discC_m}\ICinterp_{\discC_m}
c_{\rm ini}$ on $\O\times(0,t^{(1)})$. We have $\alpha_m(t)\to t$ uniformly as $m\to\infty$ and,
due to \ref{hyp:GDs}, the derivative of the inverse function $\alpha^{-1}_m$ is uniformly bounded. 
Hence, a triangle inequality, a change of variables using $\alpha_m^{-1}$, and the strong convergence of $(\Pi_{\discC_m}c_m)_{m\in\N}$
show that $\widetilde{\Pi}_{\discC_m}c_m\to c$ in $L^2(\O\times(0,T))$ as $m\to\infty$.

\subsection{Convergence of the pressure}

	\paragraph{\textbf{Step 1}: \emph{weak convergences of $\Pi_{\discP_m} p_m$ and $\grad_{\discP_m} p_m$}}
	
	We use Lemmas \ref{estimates_press} and \ref{reg-lim-space-time} to obtain ${p}\in L^\infty(0,T;H^1(\O))$ such that, up to a subsequence 
	\begin{equation}\nonumber
	\begin{aligned}
&		\Pi_{\discP_m} p_m \conv {p}\quad  \mbox{weakly-$*$ in $L^\infty(0,T;L^2(\O))$}\\
&	\grad_{\discP_m} p_m \conv \grad {p}\quad \mbox{weakly-$*$ in $L^\infty(0,T;L^2(\O)^d)$}.
	\end{aligned}
	\end{equation}

The zero-average condition in \eqref{GSpress} shows that $\int_\O \Pi_{\discP_m}p_m(\cdot,t)=0$
for all $t\in (0,T)$. Hence, the weak-$*$ convergence of $\Pi_{\discP_m}p_m$ ensures
that $\int_\O  p(\cdot,t)=0$ for a.e.\ $t\in (0,T)$ (test the zero-average condition
on $\Pi_{\discP_m}p_m$ with functions $\rho\in L^\infty(0,T)$ and pass to the limit).

Consider $\psi(\x,t)=\Xi(t)\eta(\x)$ with $\Xi\in C^\infty([0,T])$ and
$\eta\in C^\infty (\overline{\O})$. Define $\Xi_{\dt_m}(t)=\Xi(t^{(n+1)})$ on $(t^{(n)},t^{(n+1)})$ for each $n$ and note that $(\Xi_{\dt_m})_{m\in\N}$ converges to $\Xi$ uniformly.

By consistency of $(\discP_m)_{m\in\N}$, there exists $z_m\in \discP_m$ such that
 $\Pi_{\discP_m} z_m \conv \eta$ and $\grad_{\discP_m}z_m \conv \grad \eta$ strongly in $L^2(\O)$.
Recalling that $A=K/\mu$ satisfies \eqref{hyp:viscosity},
\cite[Lemma C.4]{GDMBook16} shows that $A(\x,\widetilde{\Pi}_{\discC_m}c_m)\grad_{\discP_m}z_m\rightarrow A(\x,{c})\grad \eta$ strongly in $L^2(\O\times(0,T))^d$. Apply the second equation of \eqref{GSpress} to $z=\Xi(t^{(n+1)})z_m$,
multiply by $\dtDisc_m$, and take the sum over $n=0,\dots, N_m-1$.
denoting by $q^\pm_{\dt_m}$ the piecewise-constant-in-time functions
equal to $q^\pm_n$ on $(t^{(n)},t^{(n+1)})$, we obtain
\begin{multline}\label{discrete.pressure}
\int_0^T\int_\O A(\x,\widetilde{\Pi}_{\discC_m}c_m)\nabla_{\discP_m}p_m\cdot(\Xi_{\dt_m}\nabla_{\discP_m}z_m)\\
=\int_0^T\int_\O (q^+_{\dt_m} -q^-_{\dt_m})  \Xi_{\dt_m} \Pi_{\discP_m} z_m.
\end{multline}
By symmetry of $A$, strong convergence of $\widetilde{\Pi}_{\discC_m} c_m$ and of $\grad_{\discP_m} z_m$, together with the weak convergence of $\grad_{\discP_m} p_m$, a weak--strong convergence result (see, e.g., \cite[Lemma C.3]{GDMBook16}) shows that the left-hand side of \eqref{discrete.pressure} converges to
$\int_0^T\int_{\O} A(\x,{c}) \grad {p} \cdot \Xi\grad\eta$.
Moreover, $q^\pm_{\dt_m}\to q^\pm$ in $L^1(0,T;L^2(\O))$ and thus the right-hand
side of \eqref{discrete.pressure} converges to  $\int_0^T\int_{\O} (q^+ -q^-) \Xi \eta$.
This shows that ${p}$ satisfies the second equation in \eqref{press.weak}
when $\psi=\Xi\eta$. By linear combination, this equation is also satisfied for all
tensorial functions and, by a density argument, for all smooth functions. Hence, $p$
satisfies \eqref{press.weak}.

\paragraph{\textbf{Step 2}: \emph{strong convergence of $\grad_{\discP_m} p_m$ and $\U_{\discP_m}$}}

Let $z=p_m^{(n+1)}$ in \eqref{GSpress}, multiply by $\dtDisc_m$ and take the sum over $n=0,\dots,N_m-1$.
By weak convergence of $\Pi_{\discP_m} p_m$ and since $p$ satisfies \eqref{press.weak} (which also
holds, by density, for $\psi\in L^1(0,T;H^1(\O))$),
\begin{multline*}
	\lim_{m\rightarrow\infty} \int_0^T\int_{\O} A(\x,\widetilde{\Pi}_{\discC_m} c_m) \grad_{\discP_m} p_m \cdot \grad_{\discP_m} p_m \\
=\lim_{m\rightarrow\infty}  \int_0^T\int_{\O}  (q^+_{\dt_m} -q^-_{\dt_m}) \Pi_{\discP_m} p_m 
= \int_0^T\int_{\O}  (q^{+} -q^{-}) {p}
=\int_0^T\int_\O A(\x,{c}) \grad {p}\cdot \grad{p}.
\end{multline*}
This convergence, the weak convergence of $\grad_{\discP_m} p_m$ and the strong convergence of $A(\x,\widetilde{\Pi}_{\discC_m} c_m) \grad p$ show that
\[
\begin{aligned}
	\int_0^T\int_{\O}{}& A(\x,\widetilde{\Pi}_{\discC_m} c_m)  (\grad_{\discP_m} p_m -\grad {p}) \cdot (\grad_{\discP_m} p_m-\grad {p} )\\
	={}& \int_0^T\int_{\O}  A(\x,\widetilde{\Pi}_{\discC_m} c_m) \grad_{\discP_m} p_m \cdot \grad_{\discP_m} p_m 
	-\int_0^T\int_{\O} A(\x,\widetilde{\Pi}_{\discC_m} c_m)  \grad_{\discP_m} p_m \cdot \grad {p}\\
	& -\int_0^T\int_{\O} A(\x,\widetilde{\Pi}_{\discC_m} c_m)  \grad {p} \cdot (\grad_{\discP_m} p_m -\grad {p})\conv 0.
\end{aligned}
\]
By coercivity of $A$ (Assumption \eqref{hyp:viscosity}), we infer that
$\grad_{\discP_m} p_m \rightarrow \grad {p}$ strongly in $L^2(\O\times(0,T))^d$ . Moreover, since $\grad_{\discP_m} p_m$ is bounded
in $L^\infty(0,T;L^2(\O))$ (Lemma \ref{estimates_press}), this implies that $\grad_{\discP_m} p_m \rightarrow \grad {p}$ strongly in $L^r(0,T;L^2(\O))^d$ for any $r\in(1,\infty)$.

Up to a subsequence $\widetilde{\Pi}_{\discC_m}c_m\to c$ a.e.\ on $\O\times(0,T)$.
The properties \eqref{hyp:viscosity} of $A$ and the
above convergence of $\nabla_{\discP_m}p_m$ show that $\U_{\discP_m}=-\frac{\K}{\mu(\widetilde{\Pi}_{\discC_m}c_m)}\nabla_{\discP_m}p_m\to \U=-\frac{\K}{\mu(c)}\nabla p$ strongly in $L^r(0,T;L^2(\O))^d$.

\paragraph{\textbf{Step 3}: \emph{strong convergence of $\Pi_{\discP_m} p_m$}}

Since ${p} \in L^2(0,T;H^1(\O))$, by \cite[Lemma 4.9]{GDMBook16} we can find
$P_m\in X_{\discP_m}^{N_m+1}$ such that $\Pi_{\discP_m} P_m\conv {p}$ and $\grad_{\discP_m} P_m \conv \grad {p}$ strongly in $L^2(0,T;L^2(\O))$.
Then, for each $t\in(0,T)$, by definition of the coercivity constant $C_{\discP_m}$,
\[
\norm{\Pi_{\discP_m}(P_m-p_m)}{L^2(\O)}^2 \leq C_{\discP_m}^2\left(\norm{\grad_{\discP_m}(P_m-p_m)}{L^2(\O)}^2+\left|\int_{\O} \Pi_{\discP_m}(P_m-p_m)\right|^2\right).
\]
Integrating from $0$ to $T$ and using $\int_\O {p}=\int_\O \Pi_{\discP_m} p_m =0$ yields
\begin{multline*}
\norm{\Pi_{\discP_m}(P_m-p_m)}{L^2(\O\times(0,T))}^2\\
\leq C_{\discP_m}^2\norm{\grad_{\discP_m}(P_m-p_m)}{L^2(\O\times(0,T))^d}^2+
C_{\discP_m}^2\int_0^T\left|\int_{\O} (\Pi_{\discP_m}P_m-{p})\right|^2.
\end{multline*}
The first term on the right hand side converges to 0 since both $\grad_{\discP_m} P_m$ and $\grad_{\discP_m} p_m$ converge strongly to $\grad {p}$ (and $(C_{\discP_m})_{m\in\N}$ is bounded by coercivity of $(\discP_m)_{m\in\N}$). The second term converges to 0 since $\Pi_{\discP_m} P_m$ converges to ${p}$ strongly. This shows that
$\Pi_{\discP_m} p_m$ also converges strongly to $ p$ in this space, and the
convergence in $L^r(0,T;L^2(\O))$ follows due to the  bound on $\Pi_{\discP_m}p_m$ in Lemma \ref{estimates_press}.

\subsection{Convergence of the concentration}

The proof of Theorem \ref{th:main.convergence} is concluded by showing that $c$ satisfies \eqref{conc.weak}.
It has already been established that $c\in L^2(0,T;H^1(\O))$. Lemma \ref{space_estimates_conc}
shows that $(1+|\U_{\discP_m}|)^{1/2}\nabla_{\discC_m}c_m$ is bounded in $L^2(\O\times(0,T))^d$
and therefore weakly converges, up to a subsequence, in this space to some $\mathcal W$. Since $\U_{\discP_m}$ converges
strongly in $L^2(\O\times(0,T))^d$ and $\nabla_{\discC_m}c\to \nabla c$ converges weakly in this space,
$(1+|\U_{\discP_m}|)^{1/2}\nabla_{\discC_m}c_m\to (1+|\U|)^{1/2}\nabla c$ in the sense of distributions.
Hence, $(1+|\U|)^{1/2}\nabla c=\mathcal W\in L^2(\O\times(0,T))^d$. It remains to prove that the
equation in \eqref{conc.weak} is satisfied.

Take a test function $\phy(\x,t)=\Theta(t) \omega(\x)$ with $\Theta \in C^\infty([0,T))$ and $\omega \in C^\infty(\overline{\O})$. For $m\in\N$ let $\Theta_{\dt_m}:(0,T)\to \R$ be such that
$\Theta_{\dt_m}=\Theta(t^{(n+1)})$ on $(t^{(n)},t^{(n+1)}]$ for all $n=0,\ldots,N_m-1$ (for legibility,
we drop the index $m$ in the time steps $t^{(k)}_m$).
Using Assumption \ref{hyp:smooth.interp}, define the interpolant $z_m:=\sinterp_{\discC_m} \omega$ of $\omega$.  
Now, consider $z=\Theta(t^{(n+1)}) z_m\in X_{\discC_m}$ in $\eqref{GSconc}$,
so that $v=v^{(n)}_m$ is given by $v^{(n)}_m(\x,t^{(n)})=\Theta(t^{(n+1)}) \Pi_{\discC_m} z_m(F^{(n+1)}_{t^{(n+1)}-t^{(n)}}(\x))$ (here, we make explicit the dependency on the flow $F_t^{(n+1)}$ with respect to the time
step $n$, but not with respect to $m$). Sum the resulting equations over $n=0,\ldots,N_m-1$ and recall the definition  \eqref{def.taum} of $\trans_{\darcyU_{\discP_m}}$. Letting $\hat{q}^\pm_{\dt_m}$ be the
function equal to $q_{n+1}^\pm$ on $(t^{(n)},t^{(n+1)})$ for all $n=0,\ldots,N_m-1$, we obtain
	\begin{equation}\nonumber
	\begin{aligned}
\bigg[{}&\sum_{n=0}^{N_m-1} \int_{\O} \phi \Pi_{\discC_m} c_m^{(n+1)} \Theta(t^{(n+1)})\Pi_{\discC_m} z_m - \sum_{n=0}^{N_m-1}\int_{\O} \phi \Pi_{\discC_m} c_m^{(n)} v^{(n)}_m(t^{(n)}) \bigg]\\
	&+ \int_{0}^T \!\!\!\int_{\O} \diffTens(\x,\U_{\discP_m}) \grad_{\discC_m} c_m \cdot \Theta_{\dt_m}(t) \grad_{\discC_m} z_m \\
& +  \int_0^T \!\!\!\int_{\O}  \left[w\widetilde{\Pi}_{\discC_m} c_m \trans_{\darcyU_{\discP_m}} [\Theta_{\dt_m}(t) \Pi_{\discC_m} z_m] q_{\dt_m}^-+ (1-w) \Pi_{\discC_m} c_m \Theta_{\dt_m}(t)\Pi_{\discC_m} z_m \hat{q}_{\dt_m}^-\right]  \\
	&=  \int_0^T \!\!\!\int_{\O}\left[w\trans_{\darcyU_{\discP_m}} [\Theta_{\dt_m}(t)\Pi_{\discC_m} z_m]  q_{\dt_m}^+ +(1-w) \hat{q}_{\dt_m}^+\Theta_{\dt_m}(t)\Pi_{\discC_m} z_m\right].
	\end{aligned}
	\end{equation}
Let us write $T_1^{(m)}+T_2^{(m)}+T_3^{(m)}=T_4^{(m)}$ this relation.

The limits of the last two terms are the easiest to establish.
By the strong convergences of $\Pi_{\discC_m}c_m$, $\widetilde{\Pi}_{\discC_m} c_m$ and $\Theta_{\dt_m}\Pi_{\discC_m} z_m$ in $L^2(\O\times(0,T))$, Lemma \ref{lem:conv.func.trans} shows that
\be\label{lim.T3T4}
T_{3}^{(m)} \to \int_{0}^T \int_{\O}  {c} \varphi q^-\quad\mbox{ and }\quad
T_4^{(m)} \conv  \int_0^T \int_{\O} q^+ \varphi.
\ee
Let us turn to $T_2^{(m)}$. Since $\U_{\discP_m}\to\U$ strongly in $L^2(\O\times(0,T))^d$, the growth assumption \eqref{hyp:diff}
on $\diffTens$ ensures that (see, e.g., \cite[Lemma A.1]{DT14})
\be\label{cv.DU}
\diffTens(\cdot,\U_{\discP_m})^{1/2}\to \diffTens(\cdot,\U)^{1/2}\mbox{ strongly in
$L^4(\O\times(0,T))^{d\times d}$}.
\ee
By Lemma \ref{space_estimates_conc}
the sequence $\diffTens(\cdot,\U_{\discP_m})^{1/2}\nabla_{\discC_m}c_m$ is bounded
in $L^2(\O\times(0,T))^d$. The weak convergence of $\nabla_{\discC_m}c_m$ in $L^2(\O\times(0,T))^d$
and \cite[Lemma A.3]{DT14} thus show that $\diffTens(\cdot,\U_{\discP_m})^{1/2}\nabla_{\discC_m}c_m\to
\diffTens(\cdot,\U)^{1/2}\nabla c$ weakly in $L^2(\O\times(0,T))^d$. Using
\eqref{cv.DU} and the fact that $\Theta_{\dt_m}\to \Theta$ uniformly,
the strong convergence $\nabla_{\discC_m}z_m\to \nabla\omega$ in $L^4(\O)^d$
(see \ref{hyp:smooth.interp}) shows that, as $m\to\infty$,
\begin{multline}
T_2^{(m)}=\int_{0}^T \int_{\O} \diffTens(\x,\U_{\discP_m})^{1/2} \grad_{\discC_m} c_m \cdot 
\diffTens(\x,\U_{\discP_m})^{1/2} \Theta_{\dt_m}(t) \grad_{\discC_m} z_m\\
\to
\int_{0}^T \int_{\O} \diffTens(\x,\U)^{1/2} \grad c \cdot 
\diffTens(\x,\U)^{1/2} \nabla\varphi
=\int_{0}^T \int_{\O} \diffTens(\x,\U) \grad c \cdot \nabla\varphi.
\label{lim.T2}\end{multline}

We now consider $T_1^{(m)}$. Since $\Theta(t^{(N_m)})=0$, a change of index in the first sum of $T_1^{(m)}$
and recalling the definition of $v_m^{(n)}(t^n)$ yield
		\begin{equation}\nonumber
	\begin{aligned}
	T_1^{(m)} ={}& \sum_{n=0}^{N_m-1} \int_{\O} \phi \Pi_{\discC_m} c_m^{(n)} \Theta(t^{(n)})\Pi_{\discC_m} z_m 
-\int_{\O}\phi\Pi_{\discC_m} c_m^{(0)} \Theta(t^{(0)})\Pi_{\discC_m} z_m\\
&- \sum_{n=0}^{N_m-1}\int_{\O} \phi \Pi_{\discC_m} c_m^{(n)} \Theta(t^{(n+1)})\Pi_{\discC_m}z_m\Big(F^{(n+1)}_{\dt_m^{(n+1/2)}}(\x)\Big)\\
	={}& \sum_{n=0}^{N_m-1} \int_{\O} \phi \Pi_{\discC_m} c_m^{(n)} (\Theta(t^{(n)})-\Theta(t^{(n+1)}))\Pi_{\discC_m} z_m -\int_{\O}\phi\Pi_{\discC_m} c_m^{(0)} \Theta(t^{(0)})\Pi_{\discC_m} z_m\\
	& - \sum_{n=0}^{N_m-1}\int_{\O} \phi \Pi_{\discC_m} c_m^{(n)} \Theta(t^{(n+1)})\left(\Pi_{\discC_m} z_m\Big(F^{(n+1)}_{\dt_m^{(n+1/2)}}(\x)\Big)-\Pi_{\discC_m} z_m\right)\\
	={}& T_{11}^{(m)}-T_{12}^{(m)}-T_{13}^{(m)}.
		\end{aligned}\end{equation}

	Since $c_m^{(0)}=\ICinterp_{\discC_m}c_{\rm ini}$, the consistency of $(\discC_m)_{m\in\N}$
(see Definition \ref{def:propGDs}) ensures that 
\be\label{lim.T12}
T_{12}^{(m)} \rightarrow \int_{\O} \phi {c}_{\rm {ini}} \Theta(0) \omega=\int_{\O} \phi {c}_{\rm {ini}} \varphi(\cdot,0).
\ee
Since $\Theta(t^{(n)})-\Theta(t^{(n+1)})=-\int_{t^{(n)}}^{t^{(n+1)}}\Theta'$ the strong convergences of $\Pi_{\discC_m} z_m$ and $\widetilde{\Pi}_{\discC_m} c_m$ show that
		\be\label{lim.T11}
			T_{11}^{(m)} = -\int_0^T\int_\O \phi \widetilde{\Pi}_{\discC_m}c_m \Theta' \Pi_{\discC_m}z_m
\conv -\int_{0}^T \int_{\O} \phi {c} \frac{\partial\varphi}{\partial t}.
		\ee
		It remains to analyse $T_{13}^{(m)}$. Let $\zeta_m=\Pi_{\discC_m}z_m-\omega$
and write
\[
\Pi_{\discC_m}z_m(F^{(n+1)}_{\dt^{(n+1/2)}})-\Pi_{\discC_m}z_m=\left(\omega(F^{(n+1)}_{\dt^{(n+1/2)}})-\omega\right)+\zeta_m(F^{(n+1)}_{\dt^{(n+1/2)}})-\zeta_m.
\]
Letting $\Id$ be the identity map and $\kappa(t)$ be the piecewise-constant function equal
to $\dtDisc_m$ on $(t^{(n)},t^{(n+1)})$, this yields
		\begin{multline*}
			T_{13}^{(m)} = \sum_{n=0}^{N_m-1}\int_{\O} \phi \Pi_{\discC_m} c_m^{(n)} \Theta(t^{(n+1)})\left(\omega  \Big(F^{(n+1)}_{\dt^{(n+1/2)}_m}\Big)- \omega\right)\\
			 	+\int_{t^{(1)}}^T\int_{\O}\frac{\phi}{\kappa(t)}  \widetilde{\Pi}_{\discC_m} c_m \left(\trans_{\darcyU_{\discP_m}}-\Id\right)\left[\Theta_{\dt_m}(t)\zeta_m\right] \\
			 + \Theta(t^{(1)})\int_{\O}\phi  \Pi_{\discC_m} c_m^{(0)} \left(\zeta_m(F_{t^{(1)}}^{(1)})- \zeta_m\right).
		\end{multline*}
		We note that, in the last two terms, the case $n>0$ is separated from the case $n=0$, as we do not have any information regarding the boundedness of $\grad_{\discC_m} c_m^{(0)}$ (which would arise in the estimates after invoking \ref{hyp:PiD.flow}). For a.e.\ $\x\in \O$,
$t\mapsto F^{(n+1)}_t(\x)$ is Lipschitz-continuous and the chain rule therefore yields
\begin{multline}
\omega(F^{(n+1)}_{\dt^{(n+1/2)}_m}(\x))- \omega(\x)=
-\int_{t^{(n)}}^{t^{(n+1)}}\partial_t\left[\omega (F^{(n+1)}_{t^{(n+1)}-t}(\x))\right]\\
=\int_{t^{(n)}}^{t^{(n+1)}}\grad\omega (F^{(n+1)}_{t^{(n+1)}-t}(\x))\cdot \dfrac{\discDarcyUm(F^{(n+1)}_{t^{(n+1)}-t}(\x))}{\phi((F^{(n+1)}_{t^{(n+1)}-t}(\x)))}.
\label{form.omega}
\end{multline}
The operator $\trans_{\darcyU_{\discP_m}}$ does not directly act on the time component in $L^2(\O\times(0,T))$.
Hence, the representation \eqref{def.taumstar} of its dual is also valid in $L^2(\O\times(t^{(1)},T))$,
and space-independent functions can be taken out of these operators.
Using this representation, \eqref{form.omega} and recalling the definition
\eqref{def.taum} of  $\transh_{\darcyU_{\discP_m}}$, we obtain
			\be\label{cv.T13}
\begin{aligned}
	T_{13}^{(m)}		={}& \int_0^T\int_{\O} \phi \widetilde{\Pi}_{\discC_m} c_m \transh_{\darcyU_{\discP_m}} \left[\Theta_{\dt_m}(t) \grad\omega \cdot \dfrac{\discDarcyUm[]}{\phi}\right]\\
		 	&+\int_{t^{(1)}}^T\int_{\O}\frac{\phi}{\kappa(t)}\left(\trans_{-\darcyU_{\discP_m}}-\Id\right)( \widetilde{\Pi}_{\discC_m} c_m)\Theta_{\dt_m}(t)\zeta_m\\
&+\int_{t^{(1)}}^T\int_{\O} \frac{R_m}{\kappa(t)}\trans_{-\darcyU_{\discP_m}}( \phi \widetilde{\Pi}_{\discC_m} c_m)
\Theta_{\dt_m}(t)\zeta_m\\
	 &+ \Theta(t^{(1)})\int_{\O}\phi  \Pi_{\discC_m} c_m^{(0)}\left( \zeta_m(F_{t^{(1)}}^{(1)})-\zeta_m\right)
=T_{131}^{(m)}+\cdots+T_{134}^{(m)}.
		\end{aligned}
\ee
		By weak convergence of $\Theta_{\dt_m}(t) \grad\omega \cdot \discDarcyUm[]/\phi$ (owing
to (b) in \ref{hyp:conv.u}) and strong convergence of $ \widetilde{\Pi}_{\discC_m} c_m$, Lemma \ref{lem:conv.func.trans}
shows that $T_{131}^{(m)}\conv\int_0^T \int_\O {c} \darcyU \cdot \Theta \grad \omega=\int_0^T \int_\O {c} \darcyU \cdot \grad\varphi$.
Using \ref{hyp:PiD.flow} we have, for $n=1,\ldots,N_m-1$,
\[
\frac{\norm{\Pi_{\discC_m}c_m^{(n)}(F^{(n+1)}_{-\dt_m^{(n+1/2)}})-\Pi_{\discC_m}c_m^{(n)}}{L^1(\O)}}{\dt_m^{(n+1/2)}}
\le M_F \norm{\discDarcyUm}{L^2(\O)}\norm{\nabla_{\discC_m}c_m^{(n)}}{L^2(\O)}.
\]
Hence, invoking \ref{hyp:GDs},
\begin{align*}
|T_{132}^{(m)}|\le{}&\phi^* M_F \norm{\zeta_m}{L^\infty(\O)} \norm{\Theta}{L^\infty(0,T)} \\
&\qquad\times \sum_{n=1}^{N_m-1} \dt_m^{(n+1/2)}\norm{\discDarcyUm}{L^2(\O)}\norm{\nabla_{\discC_m}c_m^{(n)}}{L^2(\O)}\\
\le{}&\phi^* M_F\norm{\zeta_m}{L^\infty(\O)} \norm{\Theta}{L^\infty(0,T)} 
\norm{\discDarcyUm[]}{L^\infty(0,T;L^2(\O))}\\
&\qquad\times M_t\sum_{n=0}^{N_m-2} \dt_m^{(n+1/2)}\norm{\nabla_{\discC_m}c_m^{(n+1)}}{L^2(\O)}\\
\le{}&\phi^* M_F\norm{\zeta_m}{L^\infty(\O)} \norm{\Theta}{L^\infty(0,T)} 
\norm{\discDarcyUm[]}{L^\infty(0,T;L^2(\O))}\norm{\nabla_{\discC_m}c_m}{L^1(0,T;L^2(\O))}.
\end{align*}
Using the bounds on $\discDarcyUm[]$ and $\nabla_{\discC_m}c_m$ given by
(a) in \ref{hyp:conv.u} and Lemmas \ref{estimates_press} and \ref{space_estimates_conc},
and the convergence $\zeta_m=\Pi_{\discC_m}\sinterp_{\discC_m}\omega-\omega\to 0$ in $L^\infty(\O)$
from \ref{hyp:smooth.interp}, we infer that $T_{132}^{(m)}\conv 0$.
The term $T_{133}^{(m)}$ also converges to 0, due to the bound on $R_m$ in Lemma \ref{lem:conv.func.trans}
(which cancels out the term $1/\kappa(t)$), the bound  \eqref{est.L2.trans} and the
convergence of $\zeta_m$ to $0$ in $L^\infty(\O)$.

Finally, let us study $T_{134}^{(m)}$. Since $\Pi_{\discC_m}c_m^{(0)}=\Pi_{\discC_m}\ICinterp_{\discC_m}c_{\rm ini}$ 
is bounded in $L^\infty(\O)$ (see Definition \ref{def:propGDs}), there is $C$ not depending
on $m$ such that $|\Theta(t^{(1)})\Pi_{\discC_m}c_m^{(0)}|\le C$ a.e.\ on $\O$.
Split $\zeta_m=\Pi_{\discC_m}z_m-\omega$ and write, using \ref{hyp:PiD.flow} on $z_m$ and Lemma
\ref{lem:flow.cont} on $\omega$,
\begin{align*}
|T_{134}^{(m)}|\le{}& C \left(\norm{\Pi_{\discC_m}z_m(F^{(1)}_{t^{(1)}})-\Pi_{\discC_m}z_m}{L^1(\O)}+
\norm{\omega(F^{(1)}_{t^{(1)}})-\omega}{L^1(\O)}\right)\\
\le{}&C\norm{\discDarcyUm[(1)]}{L^2(\O)} |\dt_m^{(1)}|\left(M_F \norm{\nabla_{\discC_m}z_m}{L^2(\O)}
+\frac{C_1(T)}{\phi_*}\norm{\nabla\omega}{L^2(\O)}\right).
\end{align*}
The bounds on $\discDarcyU[(1)]$ (from (a) in \ref{hyp:conv.u} and Lemma \ref{estimates_press})
and on $\nabla_{\discC_m}z_m$ (from \ref{hyp:smooth.interp}) then show that $T_{134}^{(m)}\conv 0$.

Hence, $T_{13}^{(m)}\to \int_0^T \int_\O {c} \darcyU \cdot \grad\varphi$.
Together with \eqref{lim.T12} and \eqref{lim.T11}, this shows that
\[
T_{1}^{(m)} \conv  -\int_0^T\int_\O \phi {c} \frac{\partial\varphi}{\partial t}
-\int_\O \phi{c}_{\rm ini} \varphi(\cdot,0) -\int_0^T \int_\O {c} \darcyU \cdot \nabla\varphi.
\]
Gathering this with \eqref{lim.T3T4} and  \eqref{lim.T2}, we infer that
$c$ satisfies the equation in \eqref{conc.weak} whenever $\phy=\Theta\omega$.
 By linear combination, this equation is also satisfied for all
tensorial functions and, by density argument, for all smooth functions. This concludes the proof
that $c$ satisfies \eqref{conc.weak}.

\section{Conclusion}\label{sec:conclusion}

We analysed the convergence of numerical schemes for a coupled elliptic--parabolic system modelling
the miscible displacement of a flow by another in a porous medium. The advective terms were discretised
by the Eulerian--Lagragian Localised Adjoint Method (ELLAM), and the diffusive terms by the generic
framework of the Gradient Discretisation Method (GDM). As a consequence, our analysis applies to a wide
range of schemes, given the variety of numerical methods for diffusion problems that fit into 
the GDM. In particular, our results apply to MFEM--ELLAM of \cite{WLELQ-00} and to the HMM--ELLAM
of \cite{CD17}. The GDM--ELLAM framework also gives an easy way to construct further ELLAM-based schemes,
by discretising the diffusion terms using any of the method known to fit into the GDM.

Contrary to previous convergence analysis of schemes involving the ELLAM, the analysis here relies neither on
$L^\infty$ bounds on the concentration (which, given the anisotropic diffusive terms and generic
meshes used in reservoir engineering, would not hold at the discrete level), nor on 
the smoothness of the data or the solutions (which cannot be established in practical situations,
with discontinuous data such as the permeability, porosity, etc.). The convergence is established
under minimal regularity assumptions on the data, using energy estimates and discrete compactness
techniques.

To carry out this analysis, fine properties of the flow of possibly discontinuous Darcy velocities
have been established. These properties, as well as some other techniques we develop for the analysis,
could certainly prove useful for other characteristic-based discretisations (such as the Modified Method
Of Characteristics).

\section{Appendix: generic compactness results}\label{sec:appen:GS}

The following results are particular cases of more general theorems on GDM that can be found in \cite{GDMBook16}.

\begin{lemma}[Regularity of the limit, space-time problems {\cite[Lemma 4.7]{GDMBook16}}] \label{reg-lim-space-time}
~\\
	Let $p\in  (1,\infty)$, and $((\disc^T)_m)_{m\in\N}$ be a coercive and limit-conforming sequence of space-time GDs. For each $m\in\N$, take $u_m\in X_{\disc_m}^{N_m+1}$ (identified with
a piecewise-constant function $[0,T]\to X_{\disc_m}$) and assume that $(\norm{u_m}{L^p(0,T;X_{\disc_m})})_{m\in\N}$ is bounded. Then there exists $u\in L^p(0,T;H^1(\O))$ such that, up to a subsequence as $m\to\infty$, $\Pi_{\disc_m}u_m \rightarrow u$ and 
	$\grad_{\disc_m}u_m \rightarrow \grad u$ weakly in $L^p(0,T;L^2(\O))$.
The same property holds with $p=+\infty$, provided that the weak convergences are replaced by weak-$*$ convergences.
\end{lemma}

\begin{definition}[Compactly--continuously embedded sequence]\label{def.cc}
	Let $(X_m,\norm{\cdot}{X_m})_{m\in\N}$ be a sequence of Banach spaces included in $L^2(\O)$, and $(Y_m,\norm{\cdot}{Y_m})_{m\in\N}$ be a sequence of Banach spaces. The sequence $(X_m,Y_m)_{m\in\N}$ is compactly--continuously embedded in $L^2(\O)$ if:
	\begin{enumerate}
		\item\label{cc.it1} If $u_m\in X_m$ for all $m\in\N$ and $(\norm{u_m}{X_m})_{m\in\N}$ is bounded, then $(u_m)_{m\in\N}$
is relatively compact in $L^2(\O)$.  
		\item\label{cc.it2} $X_m\subset Y_m$ for all $m\in\N$ and for any sequence $(u_m)_{m\in\N}$ such that 
		\begin{enumerate}
			\item $u_m\in X_m$ for all $m\in\N$ and $(\norm{u_m}{X_m})_{m\in\N}$ is bounded,
			\item $\norm{u_m}{Y_m}\conv 0$ as $m \conv \infty$,
			\item $(u_m)_{m\in\N}$ converges in $L^2(\O)$,
		\end{enumerate}
	it holds that $u_m\conv 0$ in $L^2(\O)$.
	\end{enumerate}
\end{definition}

\begin{theorem}[Discrete Aubin--Simon compactness {\cite[Theorem 4.17]{GDMBook16}}] \label{disc-Aubin-Simon}
	~\\Let $(X_m, Y_m)_{m\in\N}$
	be compactly--continuously embedded in $L^2(\O)$, $T > 0$ and $(f_m)_{m\in\N}$ be a sequence in $L^2(0,T;L^2(\O))$ such that
	\begin{itemize}
		\item For all $m \in N$, there exists $N\in \N^*$, $0=t^{(0)}<\dots <t^{(N)}=T$
and $(v^{(n)})_{n=0,\dots,N} \in X_{m}^{N+1}$ such that $f_m(t)=v^{(n+1)}$
for all $n=0,\dots,N-1$ and a.e.\ $t\in(t^{(n)},t^{(n+1)}), f_m(t)=v^{(n+1)}$.
We then set
		\[
		\delta_m f_m(t)= \dfrac{v^{(n+1)}-v^{(n)}}{t^{(n+1)}-t^{(n)}} \mbox{ for $n=0,\dots,N-1$ and $t\in(t^{(n)},t^{(n+1)})$}.
		\] 
		\item  The sequence $(f_m)_{m\in\N}$ is bounded in $L^2(0, T;L^2(\O))$.
		\item  The sequence $(\norm{f_m}{L^2(0,T;X_m)})_{m\in\N}$ is bounded.
		\item The sequence $(\norm{\delta_mf_m}{L^2(0,T;Y_m)})_{m\in\N}$ is bounded.
	\end{itemize}
	Then $(f_m)_{m\in\N}$ is relatively compact in $L^2(0,T;L^2(\O))$.
\end{theorem}

\thanks{\textbf{Acknowledgement}: this research was supported by the Australian Government through the Australian Research Council's Discovery Projects funding scheme (pro\-ject number DP170100605).
}

\bibliographystyle{abbrv}
\bibliography{ELLAM-estimates}
\end{document}

%% file: macros.tex
\allowdisplaybreaks

\definecolor{labelkey}{rgb}{0.6,0,1}

\usepackage[normalem]{ulem}
\newcounter{corr}
\definecolor{violet}{rgb}{0.580,0.,0.827}
\newcommand{\corr}[3]{\typeout{Warning : a correction remains in page
\thepage}
				\stepcounter{corr}        
				{\color{blue}\ifmmode\text{\,\sout{\ensuremath{#1}}\,}\else\sout{#1}\fi}
        {\color{red}#2}
        {\color{violet} \fbox{\thecorr}#3}}

\newcounter{cst}
\catcode`\@=11
\def\ctel#1{C_{\refstepcounter{cst}\@bsphack
\protected@write\@auxout{}%
           {\string\newlabel{#1}{{\thecst}{\thepage}}}\thecst}}
\catcode`\@=12

\newcounter{cexp}
\catcode`\@=11
\def\terml#1{T_{\refstepcounter{cexp}\@bsphack
\protected@write\@auxout{}%
           {\string\newlabel{#1}{{\thecexp}{\thepage}}}\thecexp}}
\catcode`\@=12

\newcommand{\mathbi}[1]{{\boldsymbol #1}}

\newcommand{\eop}{{\unskip\nobreak\hfil\penalty50
           \hskip2em\hbox{}\nobreak\hfil\mbox{\rule{1ex}{1ex} \qquad}
   \parfillskip=0pt
   \finalhyphendemerits=0\par\medskip}}
\renewenvironment{proof}[1][]{\noindent {\bf Proof#1. } }{\eop}

\newtheorem{theorem}{Theorem}[section]
\newtheorem{remark}[theorem]{Remark}
\newtheorem{lemma}[theorem]{Lemma} 
\newtheorem{definition}[theorem]{Definition}

\definecolor{shadecolor}{gray}{0.92}
\definecolor{TFFrameColor}{gray}{0.92}
\definecolor{TFTitleColor}{rgb}{0,0,0}

\newcommand{\ba}{\begin{array}{llll}   }
\newcommand{\bac}{\begin{array}{c}}
\newcommand{\bari}{\begin{array}{r}}
\newcommand{\ea}{\end{array}}

\newcommand{\ban}{\begin{array}{llll}}
\newcommand{\ean}{\end{array}}

\newcommand{\be}{\begin{equation}}
\newcommand{\ee}{\end{equation}}

\newcommand{\beqsys }{\beqtab \left \{ \begin{array}{l}}
\newcommand{\eeqsys }{\end{array} \right . \eeqtab }

\newcommand{\benum}{\begin{enumerate}}
\newcommand{\eenum}{\end{enumerate}}

\newcommand{\beqtab}{\begin{eqnarray}} 
\newcommand{\eeqtab}{\end{eqnarray}}

\newcommand{\dsp}{\displaystyle}

\newcommand{\combineln}[2]{{\rm(\textbf{#1#2})}}

\newcommand{\bfn}{\mathbf{n}}

\newcommand{\bfv}{\mathbi{v}}
\newcommand{\bfw}{\mathbi{w}}
\newcommand{\bfV}{\mathbi{V}}
\newcommand{\bfVh}{\bfV_{\!h,0}}
\newcommand{\bxi}{\mathbi{\xi}}

\newcommand{\bphi}{\mathbi{\phi}}

\newcommand{\bvarphi}{\mathbi{\varphi}}

\newcommand{\centercv}{\x_\cv}                         
\newcommand{\centers}{\mathcal{P}}
\newcommand{\conv}{\rightarrow} 
\newcommand{\cv}{K}

\renewcommand{\d}{{\rm d}}
\newcommand{\darcyU}{\mathbf{u}}
\newcommand{\Vel}{\mathbf{V}}
\newcommand{\dcvedge}{d_{\cv,\edge}}

\newcommand{\diffTens}{\mathbf{D }} 
\newcommand{\disc}{{\mathcal D}}
\newcommand{\discC}{{\mathcal C}}
\newcommand{\discP}{{\mathcal P}}

\newcommand{\discDarcyU}[1][(n+1)]{\darcyU_{\discP}^{#1}}
\newcommand{\discDarcyUm}[1][(n+1)]{\darcyU_{\discP_m}^{#1}}

\newcommand{\dt}{{\delta\!t}}
\newcommand{\dtDisc}{\dt^{(n+\frac{1}{2})}}

\renewcommand{\div}{{\mathop{\rm div}}}
\newcommand{\divg}{\div}
\newcommand{\proj}{E}

\def\Proj{\mathrm{Pr}}

\def\dbdr{d\mathrm{s}}

\newcommand{\edge}{\sigma}

\newcommand{\edges}{{\mathcal E}}              
\newcommand{\edgescv}{{{\edges}_\cv}}

\newcommand{\edgesint}{{\edges}_{\rm int}} 
\newcommand{\eps}{\varepsilon}

\newcommand{\grad}{\nabla}

\newcommand{\ICinterp}{\mathcal{I}}
\newcommand{\sinterp}{\mathcal{J}}

\newcommand{\mesh}{{\mathcal M}}

\newcommand{\N}{\mathbb N}

\newcommand{\ncvedge }{\bfn_{\cv,\edge}}  

\newcommand{\norm}[2]{\| #1 \|_{#2}}

\renewcommand{\O}{\Omega}

\newcommand{\phy}{\varphi}

\newcommand{\polyd}{{\mathfrak{T}}}

\newcommand{\R}{\mathbb R}

\newcommand{\snorm}[2]{\left\vert #1 \right\vert_{#2}}

\newcommand{\vertex}{{\mathbi{s}}}

\newcommand{\x}{\mathbi{x}}

\newcommand{\centeredge}{\overline{\mathbi{x}}_\edge} 
\newcommand{\y}{\mathbi{y}}

\renewcommand{\norm}[2]{\left\Vert#1\right\Vert_{#2}}

\def\ograd{\overline{\grad}}

\def\PiD{\Pi_\disc}
\def\PiDc{\Pi_{\discC}}
\def\PiDp{\Pi_{\discP}}
\def\gradD{\nabla_\disc}
\def\gradDc{\nabla_{\discC}}
\def\gradDp{\nabla_{\discP}}

\usepackage{xparse}
\DeclareDocumentCommand{\RPiD}{ O{\disc} O{,0} }{\Pi_{#1}(X_{#1#2})}

\def\Fdof#1{{\bm{\mathcal{F}}(#1,\R)}}
\makeatletter
\def\Fdof{\@ifnextchar[{\@with}{\@without}}
\def\@with[#1]#2{{\bm{\mathcal{F}}(#2;#1)}}
\def\@without#1{{\bm{\mathcal{F}}(#1,\R)}}
\makeatother

\newcommand{\U}{\mathbf{U}}
\newcommand{\Daru}{\mathbf{u}}

\def\K{\mathbf{K}}

\def\trans{\mathcal{T}}
\def\transh{\widehat{\mathcal{T}}}
\def\Id{\mathrm{I}}

\def\RT0{\mathbb{RT}_0}

%% file: fig-cell.pdf_t
\begin{picture}(0,0)%
\includegraphics{fig-cell.pdf}%
\end{picture}%
\setlength{\unitlength}{3947sp}%
\begingroup\makeatletter\ifx\SetFigFont\undefined%
\gdef\SetFigFont#1#2#3#4#5{%
  \reset@font\fontsize{#1}{#2pt}%
  \fontfamily{#3}\fontseries{#4}\fontshape{#5}%
  \selectfont}%
\fi\endgroup%
\begin{picture}(2406,1940)(1003,-2193)
\put(3272,-1768){\makebox(0,0)[lb]{\smash{{\SetFigFont{10}{12.0}{\rmdefault}{\mddefault}{\updefault}{\color[rgb]{0,0,0}$\edge$}%
}}}}
\put(2551,-1636){\makebox(0,0)[lb]{\smash{{\SetFigFont{10}{12.0}{\rmdefault}{\mddefault}{\updefault}{\color[rgb]{0,0,0}$D_{K,\edge}$}%
}}}}
\put(2026,-1111){\makebox(0,0)[lb]{\smash{{\SetFigFont{10}{12.0}{\rmdefault}{\mddefault}{\updefault}{\color[rgb]{0,0,0}$d_{K,\edge}$}%
}}}}
\put(1426,-1936){\makebox(0,0)[lb]{\smash{{\SetFigFont{10}{12.0}{\rmdefault}{\mddefault}{\updefault}{\color[rgb]{0,0,0}$K$}%
}}}}
\put(3196,-1149){\makebox(0,0)[lb]{\smash{{\SetFigFont{10}{12.0}{\rmdefault}{\mddefault}{\updefault}{\color[rgb]{0,0,0}$\bfn_{K,\edge}$}%
}}}}
\put(3144,-1520){\makebox(0,0)[lb]{\smash{{\SetFigFont{10}{12.0}{\rmdefault}{\mddefault}{\updefault}{\color[rgb]{0,0,0}$\centeredge$}%
}}}}
\put(1951,-1711){\makebox(0,0)[lb]{\smash{{\SetFigFont{10}{12.0}{\rmdefault}{\mddefault}{\updefault}{\color[rgb]{0,0,0}$\x_K$}%
}}}}
\end{picture}%